\documentclass[10pt]{amsart}
\usepackage{amsmath}
\usepackage{amsfonts}
\usepackage{amssymb}

\newcommand\eps{\varepsilon}
\newcommand\R{{\mathbf{R}}}
\newcommand\C{{\mathbf{C}}}
\newcommand\Z{{\mathbf{Z}}}

\newcommand\N{{\mathcal{N}}}
\renewcommand\S{{\mathcal{S}}}

\newcommand\supp{{\operatorname{Supp}}}
\newcommand\re{{\operatorname{Re}}}
\theoremstyle{plain}
  \newtheorem{theorem}[subsection]{Theorem}

  \newtheorem{lemma}[subsection]{Lemma}
  \newtheorem{corollary}[subsection]{Corollary}

\theoremstyle{remark}
  \newtheorem{remark}[subsection]{Remark}

\theoremstyle{definition}
  \newtheorem{definition}[subsection]{Definition}

\begin{document}

\title[]{The linear profile decomposition for the Airy equation and the existence of maximizers for the Airy Strichartz inequality}
\author{Shuanglin Shao}
\address{Department of Mathematics, UCLA, CA 90095}
\curraddr{School of Mathematics, Institute for Advanced Study, Princeton, NJ 08540}
\email{slshao@math.ias.edu}


\vspace{-0.1in}
\begin{abstract}
In this paper, we establish the linear profile decomposition for the Airy equation with complex or real initial data in $L^2$, respectively. As an application, we obtain a dichotomy result on the existence of maximizers for the symmetric Airy-Strichartz inequality.
\end{abstract}

\maketitle

\section{Introduction}
In this paper, we consider the problem of the linear profile decomposition for the Airy equation with the $L^2$ initial data
\begin{equation}\label{eq:airy}
\begin{cases}&\partial_t u+\partial_x^3 u=0, t\in \R,\, x\in \R,\\
&u(0,x)=u_0(x)\in L^2,
\end{cases}
\end{equation}
where $u: \R\times\R\to \R \text{ or }\C$. Roughly speaking, the profile decomposition is to investigate the general structure of  a sequence of solutions to the Airy equation with bounded initial data in $L^2$. We expect that it can be expressed, up to a subsequence, as a sum of a superposition of concentrating waves-- profiles--and a reminder term. The profiles are ``almost orthogonal" in the Strichartz space and in $L^2$ while the remainder term is small in the same Strichartz norm and can be negligible in practice. The profile decomposition is also referred to as the ``bubble decomposition" in the literature, see \cite[p.35]{Killip-Visan:2008:clay-lecture-notes} for an interesting historical discussion.

The same problem in the context of the wave or Schr\"odinger equations has been intensively studied recently. For the wave equations, Bahouri-G\'erard \cite{Bahouri-Gerard:1999:profile-wave} established a linear profile decomposition for the energy critical wave equation in $\R^3$ (their argument can be generalized to higher dimensions). Following \cite{Bahouri-Gerard:1999:profile-wave}, Keraani \cite{Keraani:2001:profile-schrod-H^1} obtained a linear profile decomposition for energy critical Schr\"odinger equations, also see \cite{Shao:2008:maximizers-Strichartz-Sobolev-Strichartz}. For the mass critical Schr\"odinger equations, when $d=2$, Merle-Vega \cite{Merle-Vega:1998:profile-schrod} established a linear profile decomposition, similar in spirit to that in \cite{Bourgain:1998:refined-Strichartz-NLS}; Carles-Keraani \cite{Carles-Keraani:2007:profile-schrod-1d} treated the $d=1$ case, while the higher dimensional analogue was obtained by B\'egout-Vargas \cite{Begout-Vargas:2007:profile-schrod-higher-d}. In general, a nonlinear profile decomposition can be achieved from the linear case via a perturbation argument. The first ingredient of the proof  of linear profile decompositions is to start with some refined inequality: the refined Sobolev embedding or the refined Strichartz inequality. Usually establishing such refinements needs some nontrivial work. For instance, in the  Schr\"odinger case, the two dimensional improvement is due to Moyua-Vargas-Vega \cite{Moyua-Vargas-Vega:1999} involving the $X_p^q$ spaces; the one dimensional improvement due to Carles-Keraani \cite{Carles-Keraani:2007:profile-schrod-1d} using the Hausdorff-Young inequality and the weighted Fefferman-Phong inequality \cite{Fefferman:1983:uncertainty-principle}, which Kenig-Ponce-Vega \cite{Kenig-Ponce-Vega:2000:KdV} first introduced to prove their refined Strichartz inequality \eqref{eq:KPV-Strichartz} for the Airy equation; the higher dimensional refinement due to B\'egout-Vargas \cite{Begout-Vargas:2007:profile-schrod-higher-d} based on a new bilinear restriction estimate for paraboloids by Tao \cite{Tao:2003:paraboloid-restri}. Another important ingredient of the arguments is the idea of the concentration-compactness principle which aims to compensate for the defect of compactness of the Strichartz inequality, which was exploited in \cite{Bahouri-Gerard:1999:profile-wave},  \cite{Merle-Vega:1998:profile-schrod}, \cite{Carles-Keraani:2007:profile-schrod-1d} and \cite{Begout-Vargas:2007:profile-schrod-higher-d}; also see \cite{Schindler-Tintalev:2002:abstract-concentration-compactness} for an abstract version of this principle in the Hilbert space. The profile decompositions turn to be quite useful in nonlinear dispersive equations. For instance, they can be used to analyze the mass concentration phenomena near the blow up time for the mass critical Schr\"odinger equation, see \cite{Merle-Vega:1998:profile-schrod}, \cite{Carles-Keraani:2007:profile-schrod-1d}, \cite{Begout-Vargas:2007:profile-schrod-higher-d}. It was also used to show the existence of minimal mass or energy blow-up solutions for the Schr\"odinger or wave equations at critical regularity, which is an important step in establishing the global well-posedness and scattering results for such equations, see \cite{Kenig-Merle:2006:focusing-energy-NLS-radial}, \cite{Kenig-Merle:2007:focusing-energy-nonlinear-Wave}, \cite{Killip-Tao-Visan:2008:cubic-NLS-radial}, \cite{Tao-Visan-Zhang:2007:radial-NLS-higher}, \cite{Killip-Visan:2008:focusing-energy-critical-NLS-higher-d}. In \cite{Shao:2008:maximizers-Strichartz-Sobolev-Strichartz}, the author used it to establish the existence of maximizers for the non-endpoint Strichartz and Sobolev-Strichartz inequalities for the Schr\"odinger equation.

The discussion above motivates the question of profile decompositions for the Airy equation, which is the free form of the mass critical generalized Korteweg-de Vries (gKdV) equation,
 \begin{equation}\label{eq:gKdV}
\begin{cases}&\partial_t u+\partial_x^3 u\pm u^4\partial_x u=0, t\in \R,\, x\in \R,\\
&u(0,x)=u_0(x).
\end{cases}
\end{equation}
This is one of the (generalized) KdV equations (\cite{Tao:2006-CBMS-book}) and is the natural analogy to the mass critical nonlinear Schr\"odinger equation in one spatial dimension. The KdV equations arise from describing the waves on the shallow water surfaces, and turn out to have connections to many other physical problems.   As is well known, the class of solutions to \eqref{eq:airy} enjoys a number of symmetries which preserve the mass $\int |u|^2 dx$. We will employ the notations from \cite{Killip-Tao-Visan:2008:cubic-NLS-radial} and first discuss the symmetries at the initial time $t=0$.
\begin{definition}[Mass-preserving symmetry group]For any phase $\theta\in \R/2\pi\Z$, position $x_0\in \R$ and scaling parameter $h_0>0$, we define the unitary transform $g_{\theta,x_0,h_0}:L^2\to L^2$ by the formula
$$[g_{\theta,x_0,h_0}f](x):=\frac {1}{h_0^{1/2}}e^{i\theta}f(\frac {x-x_0}{h_0}).$$
We let $G$ be the collection of such transformations. It is easy to see that $G$ is a group.
\end{definition}
Unlike the free Schr\"odinger equation
\begin{equation}\label{eq:Schrodinger}
\begin{cases}&i\partial_t u-\triangle u=0, t\in \R,\, x\in \R^d,\\
&u(0,x)=u_0(x),
\end{cases}
\end{equation}
two important symmetries are missing for $\eqref{eq:airy}$, namely, the Galilean symmetry
$$u(t,x)\mapsto e^{ix\xi_0+it|\xi_0|^2}u(t,x+2t\xi_0),$$ and the pseudo-conformal symmetry
$$u(t,x)\mapsto |t|^{-d/2}e^{-i|x|^2/(4t)}u(-1/t,x/t).$$
This lack of symmetries causes difficulties if we try to mimic the existing argument of profile decompositions for the Schr\"odinger equations. In this paper, we will show how to compensate for the lack of the Galilean symmetry when developing the analogous version of linear profile decompositions for the Airy equation \eqref{eq:airy}.

Like Schr\"odinger equations, an important family of inequalities, the Airy Strichartz inequality \cite[Theorem 2.1]{Kenig-Ponce-Vega:1991:dispersive-estimates}, is associated with the Airy equation \eqref{eq:airy}. It is invariant under the symmetry group and asserts that:
\begin{equation}\label{eq:airy-strichartz}
 \|D^{\alpha} e^{-t\partial_x^3}u_0\|_{L^q_tL^r_x} \lesssim
 \|u_0\|_{L^2},
\end{equation}
if and only if $-\alpha+\frac 3q+\frac 1r=\frac 12$ and $-1/2\le \alpha \le  1/q$, where $e^{-t\partial_x^3}u_0$ and $D^\alpha$ are defined in the ``Notation" section.
When $q=r=6$ and $\alpha=1/6$, we also have the following refined Strichartz estimate due to Kenig-Ponce-Vega, which is the key to establishing the profile decomposition results for the Airy equation in this paper.
\begin{lemma}[KPV's refined Strichartz \cite{Kenig-Ponce-Vega:2000:KdV}]\label{le:KPV-Strichartz}Let $p>1$. Then
\begin{equation}\label{eq:KPV-Strichartz}
\|D^{1/6} e^{-t\partial_x^3}u_0\|_{L^6_{t,x}}\le C\left(\sup_{\tau}|\tau|^{\frac 12-\frac 1p}\| \widehat{u_0}\|_{L^{p}(\tau)}\right)^{\frac 13}\|u_0\|^{\frac 23}_{L^2},
\end{equation} where $\tau$ denotes an interval of the real line with length $|\tau|$.
\end{lemma}
In Section \ref{sec:decomp-complx}, we will present a new proof suggested by Terence Tao by using the Whitney decomposition.

As in the Schr\"odinger case, the Airy Strichartz inequality \eqref{eq:airy-strichartz} cannot guarantee the solution map from the $L^2$ space to the Strichartz space to be compact, namely, every $L^2$-bounded sequence will produce a convergent subsequence of solutions in the Strichartz space. The particular Strichartz space we are interested  in is equipped with  the norm $\|D^{1/6}u\|_{L^6_{t,x}}$. The failure of compactness can be seen explicitly from creating counterexamples by considering the symmetries in $L^2$ such as the space and time translations, or scaling symmetry or frequency modulation. Indeed, given $x_0\in \R$, $t_0\in \R$ and $h_0\in (0,\infty)$, we denote by $\tau_{x_0}$, $S_{h_0}$ and $R_{t_0}$ the operators defined by
$$\tau_{x_0} \phi(x):=\phi(x-x_0), S_{h_0}\phi(x):=\frac {1}{h_0^{1/2}}\phi(\frac {x}{h_0}),
R_{t_0}\phi(x):=e^{-t_0\partial_x^3}\phi(x).$$
Let $(x_n)_{n\ge 1}$, $(t_n)_{n\ge 1}$ be sequences both going to infinity, and $(h_n)_{n\ge 1}$ be a sequence going to zero as $n$ goes to infinity. Then for any nontrivial $\phi\in \S$, $(\tau_{x_n}\phi)_{n\ge 1}$, $(S_{h_n}\phi)_{n\ge 1}$ and $(R_{t_n}\phi)_{n\ge 1}$ weakly converge to zero in $L^2$. However, their Strichartz norms are all equal to $\|D^{1/6}e^{-t\partial_x^3}\phi\|_{L^6_{t,x}}$, which is nonzero. Hence these sequences are not relatively compact in the Strichartz spaces. Moreover, the frequency modulation also exhibits the defect of compactness: for $\xi_0\in \R$, we define $M_{\xi_0}$ via
$$M_{\xi_0}\phi(x):=e^{ix\xi_0}\phi(x).$$
Choosing $(\xi_n)_{n\ge 1}$ to be a sequence going to infinity as $n$ goes to infinity, we see that $(M_{\xi_n}\phi)_{n\ge 1}$ converges weakly to zero. However, from Remark \ref{re:airy-schr},  $\|D^{1/6}e^{-t\partial_x^3}(e^{i(\cdot)\xi_n}\phi)\|_{L^6_{t,x}}$ converges to $3^{-1/6}\|e^{-it\partial_x^2}\phi\|_{L^6_{t,x}}$, which is not zero. This shows that the modulation operator $M_{\xi_0}$ is not compact either.

It will be clear from the statements of Theorem \ref{thm:Airy-prof} and Theorem \ref{thm:Airy-prof-real} that these four symmetries in $L^2$ above are the only obstructions to the compactness of the solution map. Hence the parameter $(h_0,\xi_0,x_0,t_0)$ plays a special role in characterizing this defect of compactness; moreover, a sequence of such parameters needs to satisfy some ``orthogonality" constraint (the terminology ``orthogonality" is in the sense of Lemma \ref{le:weak-converg}.)
\begin{definition}[Orthogonality]\label{def-ortho}For $j\neq k$, two sequences $\Gamma_n^j:=(h_n^j,\xi_n^j, x_n^j,t_n^j)_{n\ge 1}$ and $\Gamma_n^k:=(h_n^k,\xi_n^k, x_n^k,t_n^k)_{n\ge 1}$ in $(0,\infty)\times \R^3$ are orthogonal if one of the following holds,
\begin{itemize}
\item $\lim_{n\to \infty}\left(\dfrac {h_n^j}{h_n^k}+\dfrac {h_n^k}{h_n^j}+h_n^j|\xi_n^j-\xi_n^k|\right)=\infty$,
\item $(h_n^j,\xi_n^j)=(h_n^k, \xi_n^k)$ and
$$\lim_{n\to \infty}\left(\dfrac{|t_n^k-t_n^j|}{(h_n^j)^3}+\dfrac {3|(t_n^k-
t_n^j)\xi_n^j|}{(h_n^j)^2}+\dfrac{|x_n^j-x_n^k+3(t_n^j-t_n^k)
(\xi_n^j)^2|}{h_n^j}\right)=\infty.$$
\end{itemize}
\end{definition}
\begin{remark} For any $\Gamma_n^j=(h_n^j,\xi_n^j, x_n^j,t_n^j)_{n\ge 1}$, it is clear that, up to a subsequence, $\lim_{n\to\infty}|h_n^j\xi_n^j|$ is either finite or infinite. For the former, we can reduce to $\xi_n^j\equiv 0$ for all n by changing profiles, see Remark \ref{re:reduction}; for the latter, the corresponding profiles exhibit Schr\"odinger behavior in some sense, see Remark \ref{re:airy-schr}. In view of this, we will group the decompositions accordingly in the statements of our main theorems below.
\end{remark}
Now we are able to state the main theorems. When the initial data to the equation \eqref{eq:airy} is complex, the following theorem on the linear Airy profile decomposition is proven in Section \ref{sec:proof-airy-prof}.
\begin{theorem}[Complex Version]\label{thm:Airy-prof}
Let $(u_n)_{n\ge 1}$ be a sequence of complex-valued functions satisfying $\|u_n\|_{L^2}\le 1$.  Then up to a subsequence, there exists a sequence of $L^2$
functions $(\phi^j)_{j\ge 1}: \R\to \C$ and a family of pairwise orthogonal sequences
$\Gamma_n^j=(h_n^j,\xi_n^j,x_n^j,t_n^j)\in (0,\infty)\times \R^3$ such that, for any $l\ge 1$,
there exists an $L^2$ function $w_n^l: \R\to \C$ satisfying
\begin{equation}\label{eq:prof}
u_n=\sum_{1\le j\le l, \xi_n^j\equiv 0 \atop \text{ or } |h_n^j\xi_n^j|\to \infty} e^{t_n^j\partial_x^3}g_n^j[e^{i(\cdot)h_n^j\xi_n^j}\phi^{j}]+w_n^l,
\end{equation}
where $g_n^j:=g_{0,x_n^j,h_n^j}\in G$ and \begin{equation}\label{eq:err}
 \lim_{l\to \infty}\lim_{n\to \infty} \|D^{1/6}e^{-t\partial_x^3}w_n^l\|_{L^6_{t,x}}=0.
\end{equation}
Moreover, for every $l\ge 1$, \begin{equation}\label{eq:almost-ortho}
 \lim_{n\to \infty} \left(\|u_n\|^2_{L^2}-\left(\sum_{j=1}^l\|\phi^j\|^2_{L^2}
 +\|w_n^l\|^2_{L^2}\right)\right)=0.
\end{equation}
\end{theorem}
When the initial sequence is of real-value, we analogously obtain the following real-version profile decomposition. Note that we can restrict the frequency parameter $\xi_n^j$ to be nonnegative.
\begin{theorem}[Real Version]\label{thm:Airy-prof-real}
Let $(u_n)_{n\ge 1}$ be a sequence of real-valued functions satisfying $\|u_n\|_{L^2}\le 1$.
Then up to a subsequence there exists
a sequence of $L^2$ functions, $(\phi^j)_{j\ge 1}$: $\R\to \C$, and a family of orthogonal sequences $\Gamma_n^j=(h_n^j,\xi_n^j,x_n^j,t_n^j)\in (0,\infty)\times [0,\infty)\times \times\R^2$ such that, for any $l\ge 1$, there exists an $L^2$ function $w_n^l$: $\R\to \R$ satisfying
\begin{equation}\label{eq:prof-real}
u_n=\sum_{1\le j\le l, \xi_n^j\equiv 0\atop \text{ or }|h_n^j\xi_n^j|\to \infty} e^{t_n^j\partial_x^3}g_n^j[\re(e^{i(\cdot)h_n^j\xi_n^j}\phi^{j})]+w_n^l,
\end{equation}
where $g_n^j:=g_{0,x_n^j,h_n^j}\in G$ and
 \begin{equation}\label{eq:err-real}
 \lim_{l\to \infty}\lim_{n\to \infty} \|D^{1/6}e^{-t\partial_x^3}w_n^l(x)\|_{L^6_{t,x}}=0.
\end{equation}
Moreover for every $l\ge 1$, \begin{equation}\label{eq:almost-ortho-real}
 \lim_{n\to \infty} \left(\|u_n\|^2_{L^2}-\left(\sum_{1\le j\le l, \xi_n^j\equiv0 \atop \text{ or } |h_n^j\xi_n^j|\to \infty}\|\re(e^{i(\cdot)h_n^j\xi_n^j}\phi^j)\|^2_{L^2}+\|w_n^l\|^2_{L^2}\right)\right)=0.
\end{equation}
\end{theorem}
When $\lim_{n\to \infty}|h_n^j\xi_n^j|=\infty$ for some $1\le j\le l$, the profile will exhibit asymptotic ``Schr\"odinger" behavior. For simplicity, we just look at the complex case.
\begin{remark}[Asymptotic Schr\"odinger behavior]\label{re:airy-schr}
 Without loss of generality, we assume $\phi^j\in \S$ with the compact Fourier support $[-1,1]$. Then
\begin{align*}
&D^{1/6}e^{-(t-t_n^j)\partial_x^3}g_n^j[e^{i(\cdot)h_n^j\xi_n^j}\phi^{j}](x)=\int e^{i(x-x_n^j)\xi+i(t-t_n^j)\xi^3}|\xi|^{1/6}(h_n^j)^{1/2}\widehat{\phi^j}(h_n^j(\xi-\xi_n^j))d\xi\\
&=(h_n^j)^{-1/2}|\xi_n^j|^{1/6} e^{i(x-x_n^j)\xi_n^j+i(t-t_n^j)(\xi_n^j)^3}\\
&\qquad \times \int e^{i[\frac {\eta(x-x_n^j+3(t-t_n^j)(\xi_n^j)^2)}{h_n^j}+\frac {\eta^3(t-t_n^j)}{(h_n^j)^3}+\frac {3\eta^2(t-t_n^j)\xi_n^j}{(h_n^j)^2}]}|1+\frac {\eta}{h_n^j\xi_n^j}|^{1/6}\phi^j(\eta)d\eta.
\end{align*}
Setting $x':=\frac {x-x_n^j+3(t-t_n^j)(\xi_n^j)^2}{h_n^j}$ and $t':=\frac {3(t-t_n^j)\xi_n^j}{(h_n^j)^2}$. Then the dominated convergence theorem yields
\begin{align*}
\|&D^{1/6}e^{-(t-t_n^j)\partial_x^3}g_n^j[e^{i(\cdot)h_n^j\xi_n^j}\phi^{j}]\|_{L^6_{t,x}} \\
&=3^{-1/6}\|\int e^{ix'\eta+it'\eta^2}e^{it'\frac {\eta^3}{3h_n^j\xi_n^j}}|1+\frac
{\eta}{h_n^j\xi_n^j}|^{1/6}\widehat{\phi^j}d\eta\|_{L^6_{t',x'}} \\
&\to_{n\to \infty} 3^{-1/6}\|e^{-it'\partial_x^2}\phi^j\|_{L^6_{t',x'}},
\end{align*} where $e^{-it\partial_x^2}$ denotes the Schr\"odinger evolution operator defined via \begin{equation*}
e^{-it\partial_x^2}f(x):=\int_\R e^{ix\xi+it|\xi|^2}\widehat{f}(\xi)d\xi.
\end{equation*} Indeed, $$\int e^{ix'\eta+it'\eta^2}e^{it'\frac {\eta^3}{3h_n^j\xi_n^j}}|1+\frac
{\eta}{h_n^j\xi_n^j}|^{1/6}\widehat{\phi^j}d\eta\to e^{-it'\partial_x^2}\phi^j(x'),\,a.e., $$ and by using \cite[Corollary, p.334]{Stein:1993} or integration by parts, $$ \left|\int e^{ix'\eta+it'\eta^2}e^{it'\frac {\eta^3}{3h_n^j\xi_n^j}}|1+\frac {\eta}{h_n^j\xi_n^j}|^{1/6}\widehat{\phi^j}d\eta\right|\le C_{\phi^j} B(t',x')$$ for $n$ large enough but still uniform in $n$. Here
\begin{equation*}
 B(t',x')=\begin{cases} (1+|t'|)^{-1/2}\le C\bigl[(1+|x'|)(1+|t'|)\bigr]^{-1/4}, &\text{for } |x'|\le 6|t'|,\\
 (1+|x'|)^{-1}\le C\bigl[(1+|x'|)(1+|t'|)\bigr]^{-1/2}, &\text{for } |x'|>6|t'|. \end{cases}
 \end{equation*}
It is easy to observe that $B\in L^6_{t',x'}$.
\end{remark}

In the next three paragraphs, we outline the proof of Theorem \ref{thm:Airy-prof} in three steps; Theorem \ref{thm:Airy-prof-real} follows similarly. Given an $L^2$-bounded sequence $(u_n)_{n\ge 1}$, at the first step, we use the refined Strichartz inequality \eqref{eq:KPV-Strichartz} and an iteration argument to obtain a preliminary decomposition decomposition for $(u_n)_{n\ge 1}$: up to a subsequence,
$$u_n=\sum_{j=1}^N f_n^j+q_n^N,$$
where $\widehat{f_n^j}$ is supported on an interval $(\xi_n^j-\rho_n^j,\xi_n^j+\rho_n^j)$ and $|\widehat{f_n^j}|\le C(\rho_n^j)^{-1/2}$, and $e^{-t\partial_x^3}q_n^N$ is small in the Strichartz norm. Then we impose the orthogonality condition on $(\rho_n^j,\xi_n^j)$: for $j\neq k$,
 \begin{equation*}
 \lim_{n\to \infty}\left(\dfrac {\rho_n^j}{\rho_n^k}+\dfrac {\rho_n^k}{\rho_n^j}+\dfrac {|\xi_n^j-\xi_n^k|}{\rho_n^j}\right)=\infty,
\end{equation*}
to re-group the decomposition.

At the second step, for each $j\in[1,N]$, we will perform a further decomposition to $f_n^j$ to extract the space and time parameters. For simplicity, we suppress all the superscripts $j$ and re-scale $(f_n)_{n\ge 1}$ to obtain $P=(P_n)_{n\ge 1}$ by setting
$$\widehat{P_n}(\cdot):=\rho_n^{1/2}\widehat{f_n}\left(\rho_n(\cdot+\rho_n^{-1}\xi_n)\right),$$
from which we can infer that each $\widehat{P_n}$ is bounded and supported on a finite interval centered at the origin. We apply the concentration-compactness argument to $(P_n)_{n\ge1 }$ to extract $(y_n^{\alpha}, s_n^\alpha)$: for any $A\ge 1$, up to a subsequence,
\begin{equation}\label{eq:intro-1}
P_n(x)=\sum_{\alpha=1}^{A} e^{-ix\rho_n^{-1}\xi_n}
e^{s^{\alpha}_n\partial^3_x}[e^{i(\cdot)\rho_n^{-1}\xi_n}\phi^{\alpha}(\cdot)](x-y_n^{\alpha})+P^A_n(x).
\end{equation}
More precisely, we will investigate the set of weak limits, $$\mathcal{W}(P):=\{w-\lim_{n\to \infty}e^{-ix\rho_n^{-1}\xi_n}e^{-s_n\partial^3_x}
[e^{i(\cdot)\rho_n^{-1}\xi_n}P_{n}(\cdot)](x+y_n)\text{ in }L^2: (y_n, s_n)\in \R^2\},$$
where the notion $``w-\lim_{n\to\infty} f_n"$ denotes, up to a subsequence, the weak limit of $(f_n)_{n\ge 1}$ in $L^2$. Note that, due to the lack of Galilean transform and the additional multiplier weight in the current Strichartz norm, it is a slight but necessary modification to the Schr\"odinger case \cite{Carles-Keraani:2007:profile-schrod-1d}, where $\mathcal{W}(P)$ is the following set
$$\{w-\lim_{n\to \infty}e^{is_{n}\partial^2_x}P_{n}(x+y_n)\text{ in }L^2: (y_n, s_n)\in \R^2\}.$$
In \eqref{eq:intro-1}, we impose the orthogonality condition on $(y_n^\alpha,s_n^\alpha)$: for $\alpha\neq \beta$,
\begin{equation*}
\lim_{n\to\infty}\left(\left|y_n^{\beta}-y_n^{\alpha}+\dfrac {3(s_n^{\beta}-s_n^{\alpha})(\xi_n)^2}{(\rho_n)^2}\right|+\left|\dfrac
{3(s_n^{\beta}-s_n^{\alpha})\xi_n}{\rho_n}\right|+\left|s_n^{\beta}-s^{\alpha}_n\right|\right)=\infty.
\end{equation*}
The error term $P^A:=(P_n^A)_{n\ge 1}$ is small in the weak sense that
\begin{equation*}\lim_{A\to \infty}\mu(P^A):=\lim_{A\to\infty}\sup \{\|\phi\|_{L^2}: \phi\in \mathcal{W}(P^A)\}=0.
\end{equation*}
Since $f_n(x)=\sqrt{\rho_n}e^{ix\xi_n}P_n(\rho_nx)$,
\begin{equation*}
f_n(x)=\sum_{\alpha=1}^{A}\sqrt{\rho_n}e^{s^{\alpha}_n\partial^3_x}[e^{i(\cdot)\rho_n^{-1}\xi_n} \phi^{\alpha}(\cdot)](\rho_nx-y_n^{\alpha})+\sqrt{\rho_n}e^{ix\xi_n}P^A_n(\rho_nx).
\end{equation*}
Let $e_n^A:=\sqrt{\rho_n}e^{ix\xi_n}P^A_n(\rho_nx)$. Now the major task is to upgrading the previous weak convergence to
$$\lim_{A\to \infty}\lim_{n\to\infty}\|D^{1/6}e^{-t\partial_x^3}e_n^A\|_{L^6_{t,x}}=0.$$
To achieve this, we will interpolate $L^6_{t,x}$ between $L^q_{t,x}$ and $L^{\infty}_{t,x}$ for some $4\le q<6$. The $L^q_{t,x}$ norm is controlled by some localized restriction estimates and the $L^{\infty}_{t,x}$ norm is expected to be controlled by $\mu(P^A)$. Unlike the Schr\"odinger case, we will distinguish the case $\lim_{n\to \infty}|\rho_n^{-1}\xi_n|=+\infty $ from $\lim_{n\to \infty}|\rho_n^{-1}\xi_n|<+\infty$ due to the additional multiplier weight in the current Strichartz norm.

The final decomposition is obtained by setting $$(h_n^j,\xi_n^j,x_n^j,t_n^j):=((\rho_n^j)^{-1},\xi_n^j,(\rho_n^j)^{-1}y_n^j,(\rho_n^j)^{-3}s_n^j)$$
and showing two orthogonality results for the profiles.

\subsection{} The second part of this paper is devoted to applying the linear profile decomposition result to the problem of the existence of maximizers for the Airy Strichartz inequality. As a corollary of Theorems \ref{thm:Airy-prof} and \ref{thm:Airy-prof-real}, we will establish a dichotomy result. Denote \begin{equation}\label{eq:airy-max}
S_{airy}^\C:=\sup\{\|D^{1/6}e^{-t\partial_x^3}u_0\|_{L^6_{t,x}}:\|u_0\|_{L^2}=
1\}, \end{equation} when $u_0$ is complex-valued; similarly we define $S_{airy}^\R$ for real-valued initial data. We are interested in determining whether there exists a maximizing function $u_0$ with $\|u_0\|_{L^2}=1$ for which
$$\|D^{1/6}e^{-t\partial_x^3}u_0\|_{L^6_{t,x}}=S_{airy}\|u_0\|_{L^2},$$
where $S_{airy}$ represents either $S_{airy}^\C$ or $S_{airy}^\R$.
The analogous question to the Schr\"odinger Strichartz inequalities was studied by Kunze
\cite{Kunze:2003:maxi-strichartz-1d}, Foschi \cite{Foschi:2007:maxi-strichartz-2d}, Hundertmark-Zharnitsky \cite{Hundertmark-Zharnitsky:2006:maximizers-Strichartz-low-dimensions}, Carneiro \cite{Carneiro:2008:sharp-strichartz-norm}, Bennett-Bez-Carbery-Hundertmark  \cite{Bennett-Bez-Carbery-Hundertmark:2008:heat-flow-of-strichartz-norm} and the author
\cite{Shao:2008:maximizers-Strichartz-Sobolev-Strichartz}. We set
\begin{equation}\label{eq:schr-max}
S_{schr}^\C:=\sup\{\|e^{-it\Delta}u_0\|_{L^6_{t,x}(\R\times \R^d)}:\|u_0\|_{L^2(\R^d)}= 1\}.
\end{equation}
The fact $S_{schr}^\C<\infty$ is due to Strichartz \cite{Strichartz:1977} which in turn had precursors in \cite{Tomas:1975:restrict}. For the problem of existence of such optimal $S_{schr}^\C$ and explicitly characterizing the maximizers, Kunze \cite{Kunze:2003:maxi-strichartz-1d} treated the $d=1$ case and showed that maximizers exist by an elaborate concentration-compactness method. When $d=1,2$, Foschi \cite{Foschi:2007:maxi-strichartz-2d} explicitly determined the best constants and showed that the only maximizers are Gaussians up to the natural symmetries associated to the Strichartz inequality by using the sharp Cauchy-Schwarz inequality and the space-time Fourier transform. Hundertmark-Zharnitsky \cite{Hundertmark-Zharnitsky:2006:maximizers-Strichartz-low-dimensions} independently obtained this result by an interesting representation formula of the Strichartz inequalities in lower dimensions. Recently, Carneiro \cite{Carneiro:2008:sharp-strichartz-norm} proved a sharp Strichartz-type inequality by following the arguments in \cite{Hundertmark-Zharnitsky:2006:maximizers-Strichartz-low-dimensions} and found its maximizers, which derives the same results in \cite{Hundertmark-Zharnitsky:2006:maximizers-Strichartz-low-dimensions} as a corollary when $d=1,2$. Very recently, Bennett-Bez-Carbery-Hundertmark \cite{Bennett-Bez-Carbery-Hundertmark:2008:heat-flow-of-strichartz-norm} offered a new proof to determine the best constants by using the method of heat-flow. In \cite{Shao:2008:maximizers-Strichartz-Sobolev-Strichartz}, the author showed that a maximizer exists for all non-endpoint Strichartz inequalities and in all dimensions by relying on the recent linear profile decomposition results for the Schr\"odinger equations. We will continue this approach for
\eqref{eq:airy-max}. Additionally, we will use a simple but beautiful idea of asymptotic embedding of a NLS
solution to an approximate gKdV solution, which was previously exploited in \cite{Christ-Colliander-Tao:2003:asymptotics-modulation-canonical-defocusing-eqs} and
\cite{Tao:2007:scattering-quartic-gKdV}. This gives that in the complex case, $S_{schr}^\C\le 3^{1/6}S_{airy}^\C$ while in the real case, $S_{schr}^\C\le 2^{1/2}3^{1/6}S_{airy}^\R$.

\begin{theorem}\label{thm:airy-max} We have the following dichotomy on the existence of maximizers for \eqref{eq:airy-max} with the complex- or real- valued initial data, respectively:
\begin{itemize}
\item In the complex case, either a maximizer is attained for \eqref{eq:airy-max}, or there exists $\phi$ of complex value satisfying $\|\phi\|_{L^2}=1$ and $S_{schr}^\C=\|e^{-it\partial_x^2}\phi\|_{L^6_{t,x}}$, and a sequence $(a_n)_{n\ge 1}$ satisfying $\lim_{n\to\infty}|a_n|=\infty$ such that
 \begin{align*}
 \lim_{n\to\infty}\|D^{1/6}e^{-t\partial_x^3}[{e^{i(\cdot)a_n}}\phi]\|_{L^6_{t,x}}
 &=S_{airy}^\C,\\
S_{schr}^\C&=3^{1/6}S_{airy}^\C.
 \end{align*}

\item In the real case, a similar statement holds; more precisely, either a maximizer is attained for \eqref{eq:airy-max}, or there exists $\phi$ of complex value satisfying $S_{schr}^\C=\dfrac {\|e^{-it\partial_x^2}\phi\|_{L^6_{t,x}}}{\|\phi\|_{L^2}}$, and a positive sequence $(a_n)_{n\ge 1}$ satisfying $\lim_{n\to\infty}a_n=\infty $ and $\lim_{n\to\infty}\|\re(e^{i(\cdot)a_n}\phi)\|_{L^2}= 1$ such that
 \begin{align*}
 \lim_{n\to\infty}\|D^{1/6}e^{-t\partial_x^3}
 \re({e^{i(\cdot)a_n}}\phi)\|_{L^6_{t,x}}&=S_{airy}^\R,\\
S_{schr}^\C&=2^{1/2}3^{1/6}S_{airy}^\R.
 \end{align*}
\end{itemize}
\end{theorem}
\begin{remark}Note that when $S_{schr}^\C=3^{1/6}S_{airy}^\C$or $S_{schr}^\C=2^{1/2}3^{1/6}S_{airy}^\R$, the explicit $\phi$ had been uniquely determined by Foschi \cite{Foschi:2007:maxi-strichartz-2d} and Hundertmark-Zharnitsky \cite{Hundertmark-Zharnitsky:2006:maximizers-Strichartz-low-dimensions} independently: they are Gaussians up to the natural symmetries enjoyed by the Strichartz inequality for the Schr\"odinger equation.
\end{remark}
This paper is organized as follows: in Section \ref{sec:notations}
we establish some notations. In Section
\ref{sec:decomp-complx}, we make a preliminary decomposition for
an $L^2$-bounded sequence $(u_n)_{n\ge 1}$ of complex
value. In Section \ref{sec:decomp-real}, we obtain similar results
for a real sequence. In Section \ref{sec:proof-airy-prof}, we
prove Theorems \ref{thm:Airy-prof} and \ref{thm:Airy-prof-real}.
In section \ref{sec:airy-maxi}, we prove Theorem
\ref{thm:airy-max}.

\textbf{Acknowledgments.} The author is grateful to Terence Tao for many helpful discussions. The author would like to thank Jincheng Jiang and Monica Visan for their comments. The author also thanks the anonymous referees for their valuable comments and suggestions, which have been incorporated into this paper.

\section{Notation}\label{sec:notations}
We use $X\lesssim Y$, $Y\gtrsim X$, or $X=O(Y)$ to denote the estimate $|X|\le
C Y$ for some constant $0<C<\infty$, which might depend on the dimension but not on the functions. If $X\lesssim Y$ and $Y\lesssim X$ we will write $X\sim Y$. If the constant $C$ depends on a special parameter, we shall denote it explicitly by subscripts.

We define the space-time norm $L^q_tL^r_x$ of $f$ on $\R\times \R$ by
$$\|f\|_{L^q_tL^r_x(\R\times\R)}:=\left(\int_{\R}\left(\int_{\R}
|f(t,x)|^{r}d\,x\right)^{q/r}d\,t\right)^{1/q},$$ with the usual modifications when $q$ or $r$ are
equal to infinity, or when the domain $\R\times\R$ is replaced by a small space-time region. When $q=r$, we abbreviate it by $L^{q}_{t, x}$. Unless specified, all the space-time integrations are taken over $\R\times \R$, and all the spatial integrations over $\R$.

We fix the notation that $\lim_{n\to \infty}$ should be understood as $\limsup_{n\to \infty}$ throughout this paper.

The spatial Fourier transform is defined via
$$\widehat{u_0}(\xi):=\int_\R e^{-ix\xi}u_0(x)dx;$$
the space-time Fourier transform is defined analogously.

The Airy evolution operator $e^{-t\partial_x^3}$ is defined via
$$e^{-t\partial_x^3}u_0(x):=\int_\R e^{ix\xi+it\xi^3}\widehat{u_0}(\xi)d\xi.$$

The spatial derivative $\partial_x^k$, $k\in \N$, the set of positive integers, is defined via the Fourier transform,
$$\widehat{\partial_x^k}(\xi)=(i\xi)^k.$$

The fractional differentiation operator $D^\alpha$, $\alpha\in \R$, is defined via
\begin{equation*}
D^\alpha f(x):=\int_\R e^{ix\xi}|\xi|^\alpha \widehat{f}(\xi)d\xi.
\end{equation*}

The inner product $\langle\cdot,\cdot\rangle_{L^2}$ in the Hilbert space $L^2$ is defined via $$\langle f,g\rangle_{L^2}:=\int_{\R}f(x)\overline{g}(x)dx,$$
where $\overline{g}$ denotes the usual complex conjugate of $g$ in the complex plane $\C$.

\section{Preliminary decomposition:~complex version}\label{sec:decomp-complx}
To begin proving Theorems \ref{thm:Airy-prof} and \ref{thm:Airy-prof-real}, we present a new proof of the refined Strichartz inequality \eqref{eq:KPV-Strichartz} based on the Whitney decomposition. The following notation is taken from \cite{Killip-Visan:2008:clay-lecture-notes}.
\begin{definition}\label{def:dyadic-cubes}
Given $j\in \Z$, we denote by $\mathcal{D}_j$ the set of all dyadic intervals in $\R$ of length $2^j$:
$$\mathcal{D}_j:=\{2^j[k,k+1):k\in \Z\}.$$
We also write $\mathcal{D}:=\cup_{j\in \Z} \mathcal{D}_j$. Given $I\in \mathcal{D}$, we define $f_I$ by $\widehat{f}_I:=\widehat{f}1_I$ where $1_I$ denotes the indicator function of $I$.
\end{definition}
Then the Whitney decomposition we need is as follows: Given two distinct $\xi,\xi'\in \R$, there is a unique  maximal pair of dyadic intervals $I\in \mathcal{D}$ and $I'\in \mathcal{D}$ such that
\begin{equation}\label{eq:dc-1}
|I|=|I'|, \operatorname{dist}(I, I')\ge 4|I|,
\end{equation}
where $\operatorname{dist}(I,I')$ denotes the distance between $I$ and $I'$, and $|I|$ denotes the length of the dyadic interval $I$. Let $\mathcal{F}$ denote all such pairs as $\xi\neq \xi'$ varies over $\R\times \R$. Then we have
\begin{equation}\label{eq:dc-2}
\sum_{(I,I')\in \mathcal{F}}1_I(\xi)1_{I'}(\xi')=1, \text{ for a.e. } (\xi,\xi')\in \R\times \R.
\end{equation}
Since $I$ and $I'$ are maximal, $\operatorname{dist}(I,I')\le 10|I|$. This shows that for a given $I\in \mathcal{D}$, there exists a bounded number of $I'$ so that $(I,I')\in \mathcal{F}$, i.e.,
\begin{equation}\label{eq:dc-3}\forall I\in \mathcal{D}, \#\{I':(I,I')\in \mathcal{F}\}\lesssim 1.
\end{equation}
\begin{proof}[Proof of Lemma \ref{le:KPV-Strichartz}]
Given $p>1$, we normalize $\sup_{\tau\in \R}|\tau|^{1/2-1/p}\|\widehat{f}\|_{L^p(\tau)}=1$. Then for all dyadic intervals $I\in \mathcal{D}$, \begin{equation}\label{eq:dc-4}
\int_I |\widehat{f}|^pd\xi\le |I|^{1-p/2}.
\end{equation}
We square the left hand side of \eqref{eq:KPV-Strichartz} and reduce to proving
\begin{equation}\label{eq:dc-5}
\left\|\int\int e^{ix(\xi-\eta)+it(\xi^3-\eta^3)}|\xi\eta|^{1/6}\widehat{f}(\xi)\overline{\widehat{f}}(\eta)d\xi d\eta\right\|_{L^3_{t,x}}\lesssim\|\widehat{f}\|^{4/3}_{L^2}.
\end{equation}
We change variables $a:=\xi-\eta$ and $b:=\xi^3-\eta^3$ and use the Hausdorff-Young inequality in both $t$ and $x$, we need to show
\begin{equation}\label{eq:dc-6}
\int\int \frac {|\xi\eta|^{1/4}|\widehat{f}(\xi)\widehat{f}(\eta)|^{3/2}}{|\xi+\eta|^{1/2}|\xi-\eta|^{1/2}}d\xi d\eta\lesssim \int |\widehat{f}|^2d\xi.
\end{equation}
By symmetries of this expression, it is sufficient to work in the region $\{(\xi,\eta):\xi\ge 0,\eta\ge 0\}$. In this case, $|\xi\eta|^{1/4}\lesssim |\xi+\eta|^{1/2}$; so we reduce to proving
\begin{equation}\label{eq:dc-7}
\int\int \frac {|\widehat{f}(\xi)\widehat{f}(\eta)|^{3/2}}{|\xi-\eta|^{1/2}}d\xi d\eta\lesssim \int |\widehat{f}|^2d\xi.
\end{equation}
In view of \eqref{eq:dc-7}, we assume $\widehat{f}\ge 0$ from now on. Then we apply the Whitney decomposition to obtain
\begin{equation}\label{eq:dc-8}
\widehat{f}(\xi)\widehat{f}(\eta)=\sum_{(I,I')\in \mathcal{F}} \widehat{f_I}(\xi)\widehat{f}_{I'}(\eta), \text{ for a. e. }
(\xi,\eta) \in \R\times \R,
\end{equation}
and \begin{equation}\label{eq:dc-9}
\forall (\xi,\eta)\in I\times I' \text{ with } (I,I')\in \mathcal{F}, |\xi-\eta|\sim |I|.
\end{equation}
Choose a slightly larger dyadic interval containing both $I$ and $I'$ but still of length comparable to $I$, still denoted by $I$, we reduce to proving
\begin{equation}\label{eq:dc-10}
\sum_{I\in \mathcal{D}}\frac {\left(\int \widehat{f_I}^{3/2}d\xi\right)^2}{|I|^{1/2}}\lesssim \int \widehat{f}~^2d\xi.
\end{equation}
To prove \eqref{eq:dc-10} we will make a further decomposition to $f_I=\sum_{n\in \Z} f_{n,I}$: for any $n\in \Z$, define $f_{n,I}$ via
$$\widehat{f_{n,I}}:=\widehat{f}1_{\{\xi: \,2^n |I|^{-1/2}\le \widehat{f}(\xi)\le 2^{n+1}|I|^{-1/2}\}}.$$
By the Cauchy-Schwarz inequality, for any $\eps>0$,
\begin{equation}\label{eq:dc-11}
\left(\int \widehat{f_I}^{3/2}d\xi\right)^2=\left(\sum_{n\in \Z} \int \widehat{f_{n,I}}^{3/2}d\xi\right)^2
\lesssim_{\eps} \sum_{n\in \Z} 2^{|n|\eps}\left(\int \widehat{f_{n,I}}^{3/2}d\xi\right)^2.
\end{equation}
Now \eqref{eq:dc-10} is an easy consequence of the following claim:
\begin{equation}\label{eq:dc-12}
 \sum_{I\in \mathcal{D}} \frac {\left(\int \widehat{f_{n,I}}^{3/2}d\xi\right)^2}{|I|^{1/2}}
 \lesssim  2^{-|n|\eps}\int \widehat{f}~^2d\xi, \text{ for some }\eps>0 .
\end{equation}
By the Cauchy-Schwarz inequality,
\begin{equation}\label{eq:dc-13}
\left(\int\widehat{f_{n,I}}^{3/2}d\xi\right)^2\lesssim \int \widehat{f_{n,I}}^2 d\xi \int \widehat{f_{n,I}}d\xi.
\end{equation}
On the one hand, when $n\ge 0$, by the Chebyshev's inequality and \eqref{eq:dc-4},
 \begin{align*}
\int \widehat{f_{n,I}}d\xi &\lesssim 2^n|I|^{-1/2}|\{\xi\in I:\widehat{f}(\xi)\ge 2^n|I|^{-1/2}\}|\\
&\lesssim 2^n|I|^{-1/2}\frac {\int_I \widehat{f}^pd\xi}{2^{np}|I|^{-p/2}}\\
&\lesssim 2^{n(1-p)}|I|^{-1/2} |I|^{p/2}|I|^{1-p/2}\\
&=2^{-|n|(p-1)}|I|^{1/2}
\end{align*}
for any $p>1$. On the other hand, when $n<0$,
\begin{equation*}
\int \widehat{f_{n,I}}d\xi \lesssim 2^n|I|^{-1/2} |I|=2^{-|n|}|I|^{1/2}.
\end{equation*}
Combining these estimates, there exists an $\eps>0$ such that
\begin{equation}\label{eq:dc-14}
\sum_{I\in \mathcal{D}} \frac {\left(\int \widehat{f_{n,I}}^{3/2}d\xi\right)^2}{|I|^{1/2}}\lesssim
2^{-|n|\eps} \sum_{I\in \mathcal{D}}\int \widehat{f_{n,I}}^{2}d\xi.
\end{equation}
Interchanging the summation order, we have
\begin{equation}\label{eq:dc-15}
\sum_{I\in \mathcal{D}}\int \widehat{f_{n,I}}^{2}d\xi=\sum_{j\in \Z}\sum_{I\in \mathcal{D}_j}\int \widehat{f}~^21_{\{\xi\in I: \widehat{f}\sim 2^{n-j/2}\}}d\xi=\int_{\R} \sum_{j: \widehat{f} \sim 2^{n-j/2}}\widehat{f}~^2 d\xi\lesssim \int\widehat{f}~^2d\xi.
\end{equation}
Then the claim \eqref{eq:dc-12} follows from \eqref{eq:dc-14} and \eqref{eq:dc-15}. Hence the proof of Lemma \ref{le:KPV-Strichartz} is complete.
\end{proof}
By using this refined Airy Strichartz inequality \eqref{eq:KPV-Strichartz}, we extract the scaling and frequency parameters $\rho_n^j$ and $\xi_n^j$ following the approach in \cite{Carles-Keraani:2007:profile-schrod-1d}.
\begin{lemma}[Complex version:~extraction of $\rho_n^j$ and $\xi_n^j$]\label{le:scale-core-complex}
Let $(u_n)_{n\ge 1}$ be a sequence of complex valued functions with $\|u_n\|_{L^2}\le 1$. Then up to a subsequence, for any $\delta>0$, there exists $N:=N(\delta)$, a family $(\rho_n^j, \xi_n^j)_{1\le j\le N\atop n\ge 1}\in (0,\infty)\times \R$ and a family $(f_n^j)_{1\le j\le N\atop n\ge 1}$ of $L^2$-bounded sequences such that,
if $j\neq k$, \begin{equation}\label{eq:ortho-0}\lim_{n\to \infty}\left(\dfrac
{\rho_n^j}{\rho_n^k}+\dfrac {\rho_n^k}{\rho_n^j}+\dfrac {|\xi_n^j-\xi_n^k|}
{\rho_n^j}\right)=\infty,
\end{equation}
for every $1\le j\le N$, there exists a compact K in $\R$ such that
\begin{equation}\label{eq:suppt-contrl}
\sqrt{\rho_n^j} |\widehat{f_n^j}(\rho_n^j\xi+\xi_n^j)|\le C_\delta1_K(\xi),
\end{equation}
and \begin{equation}\label{eq:prof-0}
u_n=\sum_{j=1}^{N}f^j_n+q_n^N,
\end{equation} which satisfies
\begin{equation}\label{eq:err-0} \|D^{\frac 16}e^{-t\partial_x^3}q_n^N\|_{L^6_{t,x}}\le \delta,\end{equation}
and
\begin{equation}\label{eq:almost-ortho-0}
\lim_{n\to\infty}\left(\|u_n\|^2_{L^2}-\left(\sum_{j=1}^{N}\|f^j_n\|^2_{L^2}+\|q_n^N\|^2_{L^2}
\right)
\right)=0.
\end{equation}
\end{lemma}
\begin{proof}
For  $\gamma_n=(\rho_n,\xi_n)\in (0,\infty)\times \R$, we define $G_n: L^2\to L^2$ by setting
$$G_n(f)(\xi):=\sqrt{\rho_n}f(\rho_n\xi+\xi_n).$$
We will induct on the Strichartz norm. If $\|D^{\frac 16}e^{-t\partial_x^3}u_n\|_{L^6_{t,x}}
\le \delta$, then there is nothing to prove. Otherwise, up to a subsequence, we have $$\|D^{\frac 16}e^{-t\partial_x^3} u_n\|_{L^6_{t,x}}> \delta.$$
On the one hand, applying Lemma \ref{le:KPV-Strichartz} with $p=\frac 43$, we see that there exists a family of intervals $I_n^1:=[\xi^1_n-\rho_n^1, \xi^1_n+\rho_n^1]$ such that
$$\int_{I^1_n} |\widehat{u_n}|^{4/3}d\xi\ge C_1 \delta^4 (\rho^1_n)^{\frac 13},$$
where $C_1$ only depends on $C$, the constant in Lemma \ref{le:KPV-Strichartz}; note that we have used $\|u_n\|_{L^2}\le 1$ here.
On the other hand, for any $A>0$, $$\int_{I^1_n\cap \{|\widehat{u_n}|>A\}}|\widehat{u_n}|^{4/3}d\xi \le A^{-\frac 23}\|\widehat{u_n}\|^2_{L^2}\le A^{-\frac 23}.$$
Let $C_\delta:=(\frac {C_1}{2})^{-3/2}\delta^{-6}$. Then from the two considerations above, we have
$$\int_{I^1_n\cap \{|\widehat{u_n}|\le C_\delta(\rho_n^1)^{-1/2}\}}|\widehat{u_n}|^{4/3}d\xi \ge \frac {C_1}{2}\delta^4 (\rho^1_n)^{1/3}.$$
From the H\"older inequality, we have
$$\int_{I^1_n\cap \{|\widehat{u_n}|\le C_\delta(\rho_n^1)^{-1/2}\}}|\widehat{u_n}|^{\frac 43}d\xi\le C_2 \left(\int_{I^1_n\cap \{|\widehat{u_n}|\le C_\delta(\rho_n^1)^{-1/2}\}}|\widehat{u_n}|^2d\xi\right)^{2/3}(|I^1_n|)^{1/3}.$$
This yields that
$$\int_{I^1_n\cap\{|\widehat{u_n}|\le C_\delta(\rho^1_n)^{-1/2}\}}|\widehat{u_n}|^{2}d\xi\ge C'\delta^6,$$
where $C'>0$ is some constant depending only on $C_1$ and $C_2$. Define $v^1_n$ and $\gamma_n^1$ by $$\widehat{v_n^1}:=\widehat{u_n}1_{I^1_n\cap\{|\widehat{u_n}|\le C_\delta(\rho^1_n)^{-1/2}\}}, \gamma_n^1:=(\rho_n^1, \xi^1_n).$$
Then $\|v^1_n\|_{L^2}\ge (C')^{1/2}\delta^3$. Also by the definition of $G$, we have
$$|G_n^1(\widehat{v_n^1})(\xi)|=|(\rho^1_n)^{1/2}\widehat{v_n^1} (\rho^1_n\xi+\xi^1_n)|\le C_\delta1_{[-1,1]}(\xi).$$ Moreover, since the supports are disjoint on the Fourier side, we have $$\|u_n\|^2_{L^2}=\|u_n-v_n^1\|^2_{L^2}+\|v_n^1\|^2_{L^2}.$$
We repeat the same argument with $u_n-v_n^1$ in place of $u_n$. At each step, the $L^2$-norm
decreases by at least $(C')^{1/2}\delta^3$. Hence after $N:=N(\delta)$ steps, we obtain that $(v_n^j)_{1\le j\le N}$ and $(\gamma_n^j)_{1\le j\le N}$ so that
\begin{align*}u_n&=\sum_{j=1}^N v_n^j+ q_n^N,\\
\|u_n\|^2_{L^2}&=\sum_{j=1}^N\|v_n^j\|^2_{L^2}+\|q_n^N\|^2_{L^2},
\end{align*}
where the latter equality is due to the disjoint Fourier supports. We have the error term estimate $$\|D^{ \frac 16}e^{-t\partial_x^3}q_n^N\|_{L^6_{t,x}}\le \delta, $$
which gives \eqref{eq:err-0}. The properties we obtain now are almost the case except for the first point of this lemma \eqref{eq:ortho-0}. To obtain it, we will re-organize the decomposition. We impose the following condition on $\gamma_n^j:=(\rho_n^j,\xi_n^j)$: $\gamma_n^j$ and $\gamma_n^k$ are orthogonal if $$\lim_{n\to \infty}\left(\frac {\rho_n^j}{\rho_n^k}+\frac {\rho_n^k}{\rho_n^j}+\frac
{|\xi_n^j-\xi_n^k|}{\rho_n^j}\right)=\infty.$$
Then we define $f_n^1$ to be a sum of those $v_n^j$ whose $\gamma_n^j$'s are not orthogonal to $\gamma_n^1$. Then taking the least $j_0\in [2,N]$ such that $\gamma_n^{j_0}$ is orthogonal to $\gamma_n^1$, we can define $f_n^2$ to be a sum of those $v_n^j$ whose $\gamma_n^j$'s are orthogonal to $\gamma_n^1$ but not to $\gamma_n^{j_0}$. Repeating this argument a finite number of times, we obtain \eqref{eq:prof-0}. This decomposition automatically gives \eqref{eq:ortho-0}. Since the supports of the functions are disjoint on the Fourier side, we also have \eqref{eq:almost-ortho-0}. Finally we want to  make sure that,  up to a subsequence, \eqref{eq:suppt-contrl} holds.

By construction, those $v_n^j$'s kept in the definition of $f_n^1$ are such that the $\gamma_n^j$'s are not orthogonal to $\gamma_n^1$, i.e., for those $j$, we have
 \begin{equation}\label{eq:loc-1}
 \lim_{n\to \infty}\frac {\rho_n^j}{\rho_n^1}+\frac {\rho_n^1}{\rho_n^j}<\infty,\quad
 \lim_{n\to \infty}\frac {|\xi_n^j-\xi_n^1|}{\rho_n^j}<\infty.
 \end{equation}
To show \eqref{eq:suppt-contrl}, it is sufficient to show that, up to a subsequence, $G_n^1(\widehat{v_n^j})$ is bounded by a compactly supported and bounded function, which will imply \eqref{eq:suppt-contrl} with $j=1$. On the one hand, by construction, $$|G_n^j(\widehat{v_n^j})|\le C_\delta1_{[-1,1]}.$$ On the other hand, we observe that \begin{align*}
G_n^1(\widehat{v_n^j})&=G^1_n(G^j_n)^{-1}G^j_n (\widehat{v_n^j}),\\
G_n^1(G_n^j)^{-1} f(\xi)&=\sqrt{\frac {\rho_n^1}{\rho_n^j}}f(\frac{\rho_n^1}{\rho_n^j}\xi+\frac {\xi_n^1-\xi_n^j}{\rho_n^j}),
\end{align*}
which yields the desired estimate for $G_n^1(\widehat{v_n^j})$ by \eqref{eq:loc-1}. Inductively we obtain \eqref{eq:ortho-0}. Hence the proof of Lemma \ref{le:scale-core-complex} is complete.
\end{proof}

The following lemma is useful in upgrading the weak convergence of error terms to the strong
convergence in the Strichartz norm in Lemma \ref{le:spa-time}.
\begin{lemma}\label{le:restriction}
We have the following two localized restriction estimates: for $9/2<q<6$ and $\widehat{G}\in L^{\infty}(B(0,R))$ for some $R>0$,
\begin{equation}\label{eq:restri-esti-1}
\|D^{1/6}e^{-t\partial_x^3}G\|_{L_{t,x}^{q}}\le
 C_{q,R}\|\widehat{G}\|_{L^{\infty}}.
\end{equation}
For the same $G$, $4\le q<6$ and $|\xi_0|\ge 10 R$,
\begin{equation}\label{eq:restri-esti-2}
  \|e^{-t\partial_x^3}(e^{i(\cdot)\xi_0}G)\|_{L_{t,x}^{q}}\le C_{q,R}  |\xi_0|^{-1/q}\|\widehat{G}\|_{L^{\infty}}.
\end{equation}
\end{lemma}
\begin{proof}Let us start with the proof of \eqref{eq:restri-esti-1}.
Let $q=2r$ with $9/4<r<3$. After squaring, we are reduced to proving
\begin{align*}&\left\|\int_{B(0,R)}\int_{B(0,R)} e^{ix(\xi_1-\xi_2)+it(\xi_1^3-\xi_2^3)}|\xi_1\xi_2|^{1/6}
\widehat{G}(\xi_1)\overline{\widehat{G}}(\xi_2)d\xi_1d\xi_2\right\|_{L^r_{t,x}}\\
&\le C_{q,R}\|\widehat{G}\|^2_{L^{\infty}(B(0,R))}.
 \end{align*}Let $s_1:=\xi_1-\xi_2$ and $s_2:=\xi_1^3-\xi_2^3$ and denote the
resulting image of $B(0,R)\times B(0,R)$ by $\Omega$ under this change of variables. Then by
using the Hausdorff-Young inequality since $r>2$, we see that the left-hand side of the inequality above is bounded by
$$C\left(\int_{\Omega}\left||\xi_1\xi_2|^{1/6}\frac{\widehat{G}(\xi_1)\widehat{G}(\xi_2)}{|\xi_1^2
-\xi_2^2|}\right|^{r'}ds_1ds_2\right)^{\frac 1{r'}}.$$
Then if we change variables back, we obtain
$$C\left(\int_{B(0,R)\times B(0,R)}\dfrac {|\xi_1\xi_2|^{r'/6}}{|\xi_1+\xi_2|^{r'-1}|\xi_1-\xi_2|^{r'-1}}
|\widehat{G}(\xi_1)\widehat{G}(\xi_2)|^{r'}d\xi_1d\xi_2\right)^{\frac 1{r'}}.$$
As in the proof of Lemma \ref{le:KPV-Strichartz}, we may assume that $\xi_1,\xi_2\ge 0$. So we have $|\xi_1\xi_2|^{\frac 12} \lesssim |\xi_1+\xi_2|$, which leads to $(\xi_1\xi_2)^{r'/6}\lesssim (\xi_1+\xi_2)^{r'/3}$ and thus
\begin{equation*}
\dfrac {|\xi_1\xi_2|^{r'/6}}{|\xi_1+\xi_2|^{r'-1}|\xi_1-\xi_2|^{r'-1}}\lesssim \dfrac
{1}{|\xi_1-\xi_2|^{\frac 53 r'-2}}+\dfrac{1}{|\xi_1+\xi_2|^{\frac 53 r'-2}}.
\end{equation*}
Then since $|\xi|^{-\frac 53r'+2}$ is locally integrable when $3/2<r'<9/5$ and
$\widehat{G}\in L^{\infty}$, we obtain \eqref{eq:restri-esti-1}.

The proof of \eqref{eq:restri-esti-2} is similar. Setting $q=2r$ with $2\le r<3$ and following the same procedure as above, we have
\begin{align*}
\|e^{-t\partial_x^3}(e^{i(\cdot)\xi_0}G)\|^2_{L^q_{t,x}}
&=\|e^{-t\partial_x^3}(e^{i(\cdot)\xi_0}G)\overline {e^{-t\partial_x^3}(e^{i(\cdot)\xi_0}G)}\|_{L^{r}_{t,x}}\\
&=\|\int e^{ix(\xi-\eta)+it[(\xi+\xi_0)^3-(\eta+\xi_0)^3]}\widehat{G}(\xi)
\overline{\widehat{G}}(\eta)d\xi d\eta\|_{L^r_{t,x}}\\
&\lesssim \left(\int\frac {|\widehat{G}(\xi)|^{r'}|\widehat{G}(\eta)|^{r'}}{|\xi-\eta|^{r'-1}|\xi+\eta+2\xi_0|^{r'-1}}d\xi d\eta
\right)^{1/{r'}}\\
&\lesssim  \left(\int\frac
{|\widehat{G}(\xi)|^{r'}|\widehat{G}(\eta)|^{r'}}{|\xi-\eta|^{r'-1}|\xi_0|^{r'-1}}d\xi d\eta\right)^{1/{r'}}\\
&\le C_{q,R}
|\xi_0|^{-1+1/{r'}}\|\widehat{G}\|^2_{L^{\infty}}\le C_{q,R}|\xi_0|^{-2/q}\|\widehat{G}\|^2_{L^{\infty}},
\end{align*}
where we have used $|\xi+\eta+2\xi_0|\sim |\xi_0|$ since $\xi,\eta\in B(0,R)$ and $|\xi_0|\ge 10 R$.
\end{proof}
In Lemma \ref{le:scale-core-complex}, we have determined the scaling and frequency parameters. Recall that from the introduction, we are left with extracting the space and time translation parameters. For this purpose, we will apply the  concentration-compactness argument. For simplicity, we present the following lemma of this kind adapted to Airy evolution but not involving the frequency and scaling parameters. The general case is similar and will be done in the next lemma.
\begin{lemma}[Concentration-Compactness]\label{le:concentr-compact}Suppose $P:=(P_n)_{n\ge 1}$ with $\|P_n\|_{L^2}\le 1$. Then up to a subsequence, there exists a sequence $(\phi^{\alpha})_{\alpha\ge 1}\in L^2$ and a family $(y_n^{\alpha},s_n^{\alpha})\in \R^2$ such that they satisfy the following constraints,
for $\alpha\neq \beta$,
\begin{equation}\label{eq:c-c-0}\lim_{n\to\infty}
\left(|y_n^{\alpha}-y_n^{\beta}|+|s_n^{\alpha}-s_n^{\beta}|\right)=\infty,\end{equation}
and for $A\ge 1$, there exists $P^A_n\in L^2$ so that
\begin{equation}\label{eq:c-c-1}
P_n(x)=\sum_{\alpha=1}^{A} e^{s^{\alpha}_n\partial^3_x}\phi^{\alpha}(x-y_n^{\alpha})+P^A_n(x),
\end{equation}
and \begin{equation*}\lim_{A\to \infty}\mu(P^A)=0,\end{equation*} where $\mu(P^A)$ is defined in the argument below; moreover we have the following almost orthogonality identity: for any $A\ge 1$,
\begin{equation}\label{eq:c-c-2}
\lim_{n\to\infty}\left(\|P_n\|^2_{L^2}-\left(\sum_{\alpha=1}^{A}\|\phi^{\alpha}\|^2_{L^2}
+\|P^A_n\|^2_{L^2}\right)\right)=0.
\end{equation}
\end{lemma}
\begin{proof}
Let $\mathcal{W}(P)$ be the set of weak limits of subsequences of $P$ in $L^2$ after the space and time translations:
$$\mathcal{W}(P):=\{w-\lim_{n\to \infty}e^{-s_n\partial^3_x}P_n(x+y_n) \text{ in } L^2:
(y_n, s_n)\in \R^2)\}.$$
We set $\mu(P):=\sup \{\|\phi\|_{L^2}: \phi\in \mathcal{W}(P)\}$. Clearly we have
$$\mu(P)\le \lim_{n\to \infty} \|P_n\|_{L^2}.$$
If $\mu(P)=0$, then there is nothing to prove. Otherwise $\mu(P)>0$, then up to a subsequence, there exists a $\phi^1\in L^2$ and a sequence $(y^1_n,s^1_n)_{n\ge 1}\in \R^2$ such that
\begin{equation}\label{eq:cc-1}
\phi^1(x)=w-\lim_{n\to \infty}e^{-s^1_{n}\partial^3_x}P_n(x+y^1_n) \text{ in } L^2,
\end{equation}
and $\|\phi^1\|_{L^2}\ge \frac12\mu(P)$. We set $P_n^1:=P_n-e^{s^1_{n}\partial^3_x}\phi^1(x-y^1_n).$
Then since $e^{-t\partial_x^3}$ is an unitary operator on $L^2$, we have
\begin{align*}
\|P_n^1\|^2_{L^2}&=\langle P_n^1, P^1_n\rangle_{L^2}\\
&=\langle P_n-e^{s^1_{n}\partial^3_x}\phi^1(x-y^1_n),P_n-e^{s^1_{n}\partial^3_x}\phi^1(x-y^1_n)\rangle_{L^2}\\
&=\langle e^{-s^1_{n}\partial^3_x}\left(P_n-e^{s^1_{n}\partial^3_x}\phi^1(x-y^1_n)\right),
e^{-s^1_{n}\partial^3_x}\left(P_n-e^{s^1_{n}\partial^3_x}\phi^1(x-y^1_n)\right)\rangle_{L^2}\\
&=\langle e^{-s^1_{n}\partial^3_x}P_n-\phi^1(x-y^1_n),
e^{-s^1_{n}\partial^3_x}P_n-\phi^1(x-y^1_n)\rangle_{L^2}\\
&=\langle e^{-s^1_{n}\partial^3_x}P_n(x+y_n^1)-\phi^1(x),
e^{-s^1_{n}\partial^3_x}P_n(x+y_n^1)-\phi^1(x)\rangle_{L^2}\\
&=\langle P_n, P_n\rangle_{L^2}
+\langle\phi^1,\phi^1\rangle_{L^2}-\langle e^{-s^1_{n}\partial^3_x}P_n(x+y_n^1) ,\phi^1\rangle_{L^2}
-\langle \phi^1, e^{-s^1_{n}\partial^3_x}P_n(x+y_n^1)\rangle_{L^2}.
\end{align*}
Taking $n\to \infty$ and using \eqref{eq:cc-1}, we see that
\begin{align*}
\lim_{n\to \infty} \left(\|P_n\|^2_{L^2}-(\|\phi^1\|^2_{L^2}+\|P_n^1\|^2_{L^2})\right)=0,\\
e^{-s^1_{n}\partial^3_x}P^1_n(x+y^1_n)\to 0, \text{ weakly in } L^2.
\end{align*}
We replace $P_n$ with $P_n^1$ and repeat the same process: if $\mu(P^1)>0$, we obtain $\phi^2$ and
$(y_n^2, s_n^2)_{n\ge 1}$ so that $\|\phi^2\|_{L^2}\ge \frac 12\mu(P^1)$ and
\begin{equation*}
\phi^2(x)=w-\lim_{n\to \infty}e^{-s^2_n\partial^3_x}P^1_n(x+y^2_n) \text{ in } L^2.
\end{equation*}
Moreover, $(y^1_n,s^1_n)_{n\ge 1}$ and $(y_n^2, s_n^2)_{n\ge 1}$ satisfy \eqref{eq:c-c-0}. Otherwise, up to a subsequence, we may assume that
$$\lim_{n\to\infty} s_n^2-s_n^1=s_0, \lim_{n\to\infty} y_n^2-y_n^1=y_0,$$
where $(s_0,y_0)\in \R^2$. Then for any $\phi\in \S$,
$$\lim_{n\to \infty}\|e^{-(s_n^2-s_n^1)\partial_x^3}\phi(x+(y_n^2-y_n^1))-
e^{-s_0\partial_x^3}\phi(x+y_0)\|_{L^2}=0.$$
That is to say, $\left(e^{-(s_n^2-s_n^1)\partial_x^3}\phi(x+(y_n^2-y_n^1))\right)_{n\ge 1}$ converges strongly in $L^2$. On the other hand, we rewrite,
$$e^{-s^2_n\partial^3_x}P^1_n(x+y^2_n)=e^{-(s_n^2-s_n^1)\partial_x^3}
\left(e^{-s^1_{n}\partial^3_x}P^1_n(x+y^1_n)\right)(x+(y_n^2-y_n^1)).$$ Now the strong convergence and weak convergence together yield $\phi^2=0$, hence $\mu(P^1)=0$, a contradiction. Hence \eqref{eq:c-c-0} holds.

Iterating this argument, a diagonal process produces a family of pairwise orthogonal sequences
$(y_n^{\alpha}, s_n^{\alpha})_{\alpha\ge 1}$ and $(\phi^{\alpha})_{\alpha\ge 1}$ satisfying
\eqref{eq:c-c-1} and \eqref{eq:c-c-2}. From \eqref{eq:c-c-2}, $\sum_{\alpha}\|\phi^{\alpha}\|^2_{L^2}$ is convergent and hence $\lim_{\alpha\to \infty}\|\phi^{\alpha}\|_{L^2}=0$. This gives $$\lim_{A\to \infty}\mu(P^A)=0,$$ since $\mu(P^A)\le 2\|\phi^A\|_{L^2}$ by construction.
\end{proof}
We are ready to extract the space and time parameters of the profiles.
\begin{lemma}[Complex version:~extraction of $x_n^{j,\alpha}$ and $s_n^{j,\alpha}$]\label{le:spa-time}
Suppose an $L^2$-bounded sequence $(f_n)_{n\ge 1}$ satisfies \begin{equation*}
\sqrt{\rho_n}|\widehat{f_n}(\rho_n(\xi+(\rho_n)^{-1}\xi_n))|\le F(\xi)
\end{equation*} with $F\in L^{\infty}(K)$ for some compact set $K$ in $\R$ independent of $n$. Then up to a subsequence, there exists a family $(y_n^{\alpha}, s_n^{\alpha})\in \R\times \R$ and a sequence $(\phi^{\alpha})_{\alpha\ge 1}$ of $L^2$ functions such that, if $\alpha\neq \beta$,
\begin{equation}\label{eq:ortho-2}
\lim_{n\to\infty}\left(\left|y_n^{\beta}-y_n^{\alpha}+\dfrac {3(s_n^{\beta}-s_n^{\alpha})(\xi_n)^2}{(\rho_n)^2}\right|+\left|\dfrac
{3(s_n^{\beta}-s_n^{\alpha})\xi_n}{\rho_n}\right|+\left|s_n^{\beta}-s^{\alpha}_n\right|\right)=\infty,
\end{equation}
and for every $A\ge 1$, there exists $e_n^A\in L^2$,
\begin{equation}\label{eq:prof-2}
f_n(x)=\sum_{\alpha=1}^{A}\sqrt{\rho_n}e^{s^{\alpha}_n\partial^3_x}[e^{i(\cdot)\rho_n^{-1}\xi_n} \phi^{\alpha}(\cdot)](\rho_nx-y_n^{\alpha})+e_n^A(x),
\end{equation}
and \begin{equation}\label{eq:err-2}
\lim_{A\to \infty}\lim_{n\to \infty}\|D^{\frac 16}e^{-t\partial_x^3}e_n^A\|_{L^6_{t,x}}=0, \end{equation}
and for any $A\ge 1$,
\begin{equation}\label{eq:almost-ortho-2}
\lim_{n\to\infty}\left(\|f_n\|^2_{L^2}-\left(\sum_{\alpha=1} ^{A}\|\phi^{\alpha}\|^2_{L^2}+\|e^A_n
\|^2_{L^2}\right)\right)=0.
\end{equation}
\end{lemma}
\begin{proof}
Setting $P:=(P_n)_{n\ge 1}$ with
$\widehat{P_n}(\xi):=\sqrt{\rho_n}\widehat{f_n}(\rho_n(\xi+(\rho_n)^{-1}\xi_n))$. Then $$\widehat{P_n}\in L^{\infty}(K).$$ Let $\mathcal{W}(P)$ be the set of weak limits in $L^2$ defined via
$$\mathcal{W}(P):=\{w-\lim_{n\to \infty}e^{-ix\rho_n^{-1}\xi_n}e^{-s_n\partial^3_x}
[e^{i(\cdot)\rho_n^{-1}\xi_n}P_{n}(\cdot)](x+y_n)\text{ in }L^2: (y_n, s_n)\in \R^2\},$$ and $\mu(P)$ as in the previous lemma. Then a similar concentration-compactness argument shows that, up to a subsequence, there exists a family $(y_n^\alpha, s_n^\alpha)_{\alpha \ge 1\atop n\ge 1}$ and $(\phi^\alpha)_{\alpha\ge 1}\in L^2$ such that \eqref{eq:ortho-2} holds, and
\begin{equation*}
P_n(x)=\sum_{\alpha=1}^{A} e^{-ix\rho_n^{-1}\xi_n}
e^{s^{\alpha}_n\partial^3_x}[e^{i(\cdot)\rho_n^{-1}\xi_n}\phi^{\alpha}(\cdot)](x-y_n^{\alpha})+P^A_n(x).
\end{equation*} As weak limits, each $\widehat{\phi^\alpha}$ has the same support as $\widehat P_n$, so does $\widehat{P_n^A}$. Furthermore, we may assume that $\widehat{\phi^\alpha}, \widehat{P_n^A}\in L^\infty(K)$.
Setting $P^A:=(P^A_n)_{n\ge 1}$. Then the sequence $(P^A)_{A\ge 1}$ satisfies \begin{equation}\label{eq:spa-time-0}\lim_{A\to \infty} \mu(P^A)=0.
 \end{equation}
For any $A\ge 1$, we also have
\begin{equation*}
\lim_{n\to\infty}\left(\|P_n\|^2_{L^2}-\left(\sum_{\alpha=1}^{A}\|\phi^{\alpha}\|^2_{L^2}+
\|P^A_n\|^2_{L^2}\right)\right)=0.
\end{equation*}
Since $f_n(x)=\sqrt{\rho_n}e^{ix\xi_n}P_n(\rho_nx)$, the decomposition \eqref{eq:prof-2}
of $f_n$ follows after setting $e_n^A(x):=\sqrt{\rho_n}e^{ix\xi_n}P^A_n(\rho_nx)$.

What remains to show is that
$$\lim_{A\to \infty}\lim_{n\to \infty}\|D^{\frac 16}e^{-t\partial_x^3}[\sqrt{\rho_n}e^{iy\xi_n}
P^A_n(\rho_ny)]\|_{L^6_{t,x}}=0,$$
which will follow from \eqref{eq:spa-time-0} and the restriction estimates in Lemma \ref{le:restriction} by an interpolation argument. Indeed, by scaling, it is equivalent to showing that
\begin{equation}\label{eq:spa-time-1}
\lim_{A\to \infty}\lim_{n\to \infty}\|D^{1/6}e^{-t\partial_x^3}[e^{i(\cdot)a_n}P_n^A]\|_{L^6_{t,x} }=0,
\end{equation}  where $a_n:=(\rho_n)^{-1}\xi_n$. Up to a subsequence, we split into two cases according to whether
$\lim_{n\to \infty} |a_n|=\infty$ or not.

$\mathbf{\emph{Case 1.}}$ $\lim_{n\to \infty}|a_n|=\infty$. By using the H\"ormander-Mikhlin multiplier theorem \cite[Theorem 4.4]{Tao:2007-247A-fourier-analysis}, for sufficiently large $n$, we have
\begin{equation*}
\|D^{1/6}e^{-t\partial_x^3}[e^{i(\cdot)a_n}P_n^A]\|_{L^6_{t,x}} \lesssim |a_n|^{1/6}\|e^{-t\partial_x^3}[e^{i(\cdot)a_n}P_n^A]\|_{L^6_{t,x}}.
\end{equation*} We will show that, after taking limits in $n$, the right hand side is bounded by $C_q\mu(P^A)^{1-q/6}$ for some $4\le q<6$. Then $\lim_{A\to \infty}\mu(P^A)=0$ yields the result. We choose a cut-off $\chi_n(t,x):=\chi_{n,1}(t)\chi_{n,2}(x)$ satisfying
\begin{equation*}\chi_{n,2}(x):=\chi_2(x)e^{ixa_n},\chi_2\in \S,
\end{equation*}
where $\widehat{\chi}_2$ is compactly supported and $\widehat{\chi}_2(\xi):=1$ on the common support $K$ of $\widehat{P_n}$, and
\begin{equation*}
\widehat{\chi}_{n,1}((\xi+a_n)^3):=\widehat{\chi}_1 (\xi^3), \chi_1\in \S,
\end{equation*} where $\widehat{\chi}_1(\xi^3):=1$ on $\supp\widehat{\chi}_2$.
Let $*$ denote the space-time convolution, then
\begin{equation}\label{eq:loc-2}
\chi_n*[e^{-t\partial_x^3}(e^{i(\cdot)a_n}P_n^A)]=e^{-t\partial_x^3}(e^{i(\cdot)a_n}P_n^A).
\end{equation}
Indeed, the space-time Fourier transform of $\chi_n$ is equal to
$$\widehat{\chi}_n(\tau,\xi):=\int e^{-it\tau-ix\xi} \chi_n(t,x)dtdx=\widehat{\chi}_2(\xi-a_n)\widehat{\chi}_{n,1}(\tau).$$
On the support of the space-time Fourier transform of $e^{-t\partial_x^3}(e^{i(\cdot)a_n}P_n^A),$ we see that $$\widehat{\chi}_n(\tau,\xi)\equiv 1.$$ This gives \eqref{eq:loc-2}. Then by the H\"older inequality and the restriction estimate \eqref{eq:restri-esti-2} in Lemma \ref{le:restriction}, for sufficiently large $n$, \begin{align*}&\|e^{-t\partial_x^3}(e^{i(\cdot)a_n}P_n^A)\|_{L^6_{t,x}}
=\|\chi_n*[e^{-t\partial_x^3}(e^{i(\cdot)a_n}P_n^A)]\|_{L^6_{t,x}} \\
&\lesssim\|\chi_n*[e^{-t\partial_x^3}(e^{i(\cdot)a_n}P_n^A)]\|^{q/6}_{L^q_{t,x}}
\|\chi_n*[e^{-t\partial_x^3}(e^{i(\cdot)a_n}P_n^A)]\|^{1-q/6}_{L^{\infty}_{t,x}}\\
&\lesssim |a_n|^{-1/6}\|F\|^{q/6}_{L^{\infty}}
 \|\chi_n*[e^{-t\partial_x^3}(e^{i(\cdot)a_n}P_n^A)]\|^{1-q/6}_{L^{\infty}_{t,x}},
\end{align*}
for some $4<q<6$. There exists $(t_n,y_n)_{n\ge 1}$ such that
$$\|\chi_n*[e^{-t\partial_x^3}(e^{i(\cdot)a_n}P_n^A)]\|_{L^{\infty}_{t,x}}\sim
\left|\chi_n*[e^{-t\partial_x^3}(e^{i(\cdot)a_n}P_n^A)](t_n,y_n)\right|.$$
We expand the right hand side out,
\begin{align*}
&\left|\int\int \chi_{n,1}(-t)\chi_{n,2}(-x)
e^{-t\partial_x^3}[e^{-t_n\partial_x^3}(e^{i(\cdot)a_n}P_n^A)(\cdot+y_n)](x)dxdt\right|.
\end{align*}
Setting $p_n(x)=e^{-t_n\partial_x^3}(e^{i(\cdot)a_n}P_n^A)(x+y_n)$, then it equals
$$\left|\int\int\widehat{\chi_1}(\eta^3)\widehat{\chi_2}(\eta)e^{-ix\eta}d\eta\,e^{-ixa_n}p_n(x)\,dx\right|
=\left|\int\chi_2(-x)\,e^{-ixa_n}p_n(x)\,dx\right|.$$
Taking $n\to \infty$, and using the definition of $\mathcal{W}(P^A)$ followed by the Cauchy-Schwarz inequality, we obtain,
\begin{equation*}
 \lim_{n\to\infty} \|\chi_n*[e^{-t\partial_x^3}(e^{i(\cdot)a_n}P_n^A)]\|_{L^{\infty}_{t,x}}
 \lesssim \|\chi_2\|_{L^2}\mu(P^A)\lesssim_{\chi_2} \mu(P^A).
\end{equation*}
Hence the claim \eqref{eq:spa-time-1} follows.

$\mathbf{\emph{Case 2.}}$ $\lim_{n\to\infty} |a_n|<\infty$. From the H\"older inequality, we have the $L^6_{t,x}$ norm in \eqref{eq:spa-time-1} is bounded by
\begin{equation*}
\|D^{1/6}e^{-t\partial_x^3}[e^{i(\cdot)a_n}P_n^A]\|^{q/6}_{L^q_{t,x}}
\|D^{1/6}e^{-t\partial_x^3}[e^{i(\cdot)a_n}P_n^A]\|^{1-q/6}_{L^\infty_{t,x}}
\end{equation*}
for some $4<q<6$. On the one hand, since $\lim_{n\to \infty}|a_n|$ is finite and $\widehat{P_n^A}\in L^{\infty}(K)$, there exists a large $R>0$ so that
$$\supp{\mathcal{F}[e^{i(\cdot)a_n}P_n^A]}\subset B(0,R),$$
where $\mathcal{F}(f)$ denotes the spatial Fourier transform of $f$. Then from \eqref{eq:restri-esti-1} in Lemma \ref{le:restriction}, we see
$$\|D^{1/6}e^{-t\partial_x^3}[e^{i(\cdot)a_n}P_n^A]\|_{L^q_{t,x}}\le C_{q,R}\|F\|_{L^\infty},$$
which is independent of $n$. On the other hand, from the Bernstein inequality, we have
$$\|D^{1/6}e^{-t\partial_x^3}[e^{i(\cdot)a_n}P_n^A]\|_{L^\infty_{t,x}}\le C_{q,R} \|e^{-t\partial_x^3}[e^{i(\cdot)a_n}P_n^A]\|_{L^\infty_{t,x}}.$$
Then a similar argument as in \emph{Case 1} shows that $\|e^{-t\partial_x^3}[e^{i(\cdot)a_n}P_n^A]\|_{L^\infty_{t,x}}$ is bounded by $\mu(P^A)^c$ for some $c>0$. Hence \eqref{eq:spa-time-1} follows and the proof of Lemma \ref{le:spa-time} is complete.
\end{proof}
\begin{remark}\label{re:reduction}
In view of the previous lemma, we will make a very useful reduction when $\lim_{n\to \infty}\rho_n^{-1}\xi_n=a$ is finite: we will take $\xi_n\equiv 0$. Indeed, we first replace $e^{i(\cdot)\rho_n^{-1}\xi_n} \phi^{\alpha}$ with $e^{i(\cdot)a}\phi^{\alpha}$ by putting the difference into the error term; then we can reduce it further by regarding $e^{i(\cdot)a}\phi^\alpha$ as a new $\phi^\alpha$.
\end{remark}
Next we will show that the profiles obtained in \eqref{eq:prof-2} are strongly decoupled under the orthogonality condition \eqref{eq:ortho-2}; more general version is in Lemma \ref{le:weak-converg}.
To abuse the notation, we denote $$\widetilde{g_n^\alpha}(\phi^{\alpha})(x):=\sqrt{\rho_n}e^{s^{\alpha}_n\partial^3_x}
[e^{i(\cdot)\rho_n^{-1}\xi_n}\phi^{\alpha}(\cdot)](\rho_n x-y_n^{\alpha}),$$
 where $\xi_n\equiv 0$ when $\lim_{n\to \infty}\rho_n^{-1}\xi_n$ is finite. \begin{corollary}\label{coro:strong-decoupling}
Under \eqref{eq:ortho-2}, for any $\alpha\neq\beta$, we have
\begin{equation}\label{eq:spa-time-2}
\lim_{n\to\infty}\left|\langle\widetilde{g_n^\alpha}(\phi^{\alpha}),
\widetilde{g_n^\beta}(\phi^{\beta})\rangle_{L^2}\right|=0
\end{equation}
and for any $1\le \alpha \le A$, \begin{equation}\label{eq:spa-time-3}
    \lim_{n\to\infty}\left|\langle\widetilde{g_n^\alpha}(\phi^{\alpha}),e_n^A\rangle_{L^2}\right|=0.
    \end{equation}
\end{corollary}
\begin{proof}Without loss of generality, we assume that $\phi^{\alpha}$ and $\phi^\beta$ are Schwartz functions with compact Fourier supports. We first prove \eqref{eq:spa-time-2}. By changing variables, we have
 \begin{align*}
    &\left|\langle\widetilde{g_n^\alpha}(\phi^{\alpha}),
   \widetilde{g_n^\beta}(\phi^{\beta})\rangle_{L^2}\right|\\
    &=\left|\langle \sqrt{\rho_n}e^{s^{\alpha}_n\partial^3_x}
  [e^{i(\cdot)\rho_n^{-1}\xi_n}\phi^{\alpha}(\cdot)](\rho_n x-y_n^{\alpha}),
  \sqrt{\rho_n}e^{s^{\beta}_n\partial^3_x}
  [e^{i(\cdot)\rho_n^{-1}\xi_n}\phi^{\beta}(\cdot)](\rho_n x-y_n^{\beta})\rangle_{L^2}\right|\\
   &=\left|\langle e^{-(s_n^\beta-s^{\alpha}_n)\partial^3_x}
  [e^{i(\cdot)\rho_n^{-1}\xi_n}\phi^{\alpha}(\cdot)]( x+y_n^{\beta}-y_n^\alpha),
  e^{ix\rho_n^{-1}\xi_n}\phi^{\beta}(x)\rangle_{L^2}\right|\\
  &\le \langle \left|\int e^{i\xi(x+y_n^\beta-y_n^\alpha+3\frac {(s_n^\beta-s_n^\alpha)\xi_n^2}{\rho_n^2})+i\xi^3(s_n^\beta-s_n^\alpha)+3i\xi^2
  \frac{(s_n^\beta-s_n^\alpha)\xi_n}{\rho_n}}\widehat{\phi^\alpha}(\xi)d\xi\right|,
  \left|\phi^\beta\right|\rangle_{L^2}.
    \end{align*}
Hence if \eqref{eq:ortho-2} holds, by using \cite[Corollary, p.334]{Stein:1993} or integration by parts combined with the dominated convergence theorem, it goes to zero as $n$ goes to infinity.

To prove \eqref{eq:spa-time-3}, we write $e_n^A=\sum_{\beta=A+1}^B \widetilde{g_n^\beta}(\phi^\beta)+e_n^B$ for any $B>A$. Recall
$$e_n^B=\sqrt{\rho_n}\left(e^{i(\cdot)\rho_n^{-1}\xi_n}P_n^B\right)(\rho_nx).$$
Then \begin{align*}&\left|\langle\widetilde{g_n^\alpha}(\phi^{\alpha}), e_n^A\rangle_{L^2}\right|
      \le \sum_{\beta=A+1}^B \left|\langle \widetilde{g_n^\alpha}(\phi^{\alpha}), \widetilde{g_n^\beta}(\phi^{\beta}) \rangle_{L^2}\right|\\
    &\qquad +\left|\langle \phi^\alpha, e^{-ix\rho_n^{-1}\xi_n}e^{-s_n^\alpha\partial_x^3}(e^{i(\cdot)\rho_n^{-1}\xi_n}P_n^B)(x+y_n^\alpha)
    \rangle_{L^2}\right|.
    \end{align*}
When $n$ goes to infinity, the first term goes to zero because of \eqref{eq:spa-time-2}. The second term is less than $\|\phi^\alpha\|_{L^2}\mu(P^B)$ by the definitions of $\mathcal{W}(P^B)$ and $\mu(P^B)$, and the Cauchy-Schwarz inequality; so it can be made arbitrarily small if taking $B$ large enough. Hence \eqref{eq:spa-time-3} is obtained by taking $B\to\infty$.
\end{proof}

\section{Preliminary decomposition:~real version}\label{sec:decomp-real}
To prove Theorem \ref{thm:Airy-prof-real}, we need the corresponding real version of lemmas in the previous section, especially of Lemma \ref{le:scale-core-complex}, \ref{le:spa-time}. To develop the real analogue of Lemma \ref{le:scale-core-complex}, we recall the following lemma due to Kenig-Ponce-Vega \cite{Kenig-Ponce-Vega:2000:KdV}.
\begin{lemma}\label{le:KPV-lemma}
Let $u_0\in L^2$ be a real-valued function with $\|u_0\|_{L^2}=1$. Then for any $\delta>0$, there are a sequence of real valued functions $f^1,\ldots, f^N$, $e^N$ and intervals $\tau_1,\ldots, \tau_N$, $N=N(\delta)\in \N$ and $C_\delta>0$, such that
\begin{equation*}
\overline{\widehat{f^j}}(\xi)=\widehat{f^j}(-\xi), \supp\widehat{f^j}\subset \tau_j\cup(-\tau_j), |\tau_j|=\rho_j,
\end{equation*}
\begin{equation*}
|\widehat{f^j}|\le C_\delta\rho_j^{-1/2},
\end{equation*}
and \begin{equation*}
u_0=\sum_{j=1}^N f^j+e^N,
\end{equation*}
with \begin{equation*}
\|u_0\|^2_{L^2}=\sum_{j=1}^N \|f^j\|^2_{L^2}+\|e^N\|^2_{L^2},
\end{equation*}
\begin{equation*}
\|D^{1/6}e^{-t\partial_x^3}e^N\|_{L^6_{t,x}}<\delta.
\end{equation*}
\end{lemma}
The proof of this lemma is similar to that of the previous Lemma \ref{le:scale-core-complex} with the help that, for real functions $f$, $\overline{\widehat{f}}(\xi)=\widehat{f}(-\xi)$. For our purpose, we will do a little more on the decomposition above. Indeed, from the proof in \cite{Kenig-Ponce-Vega:2000:KdV} we know that $\widehat{f^j}(\xi)=1_{\{\xi\in \tau_j\cup(-\tau_j):~|\widehat{u_0}|\le C_\delta\rho_j^{-1/2}\}}\widehat{u_0}(\xi)$ and $\tau_j\subset(0,\infty)$. We can decompose $f^j$ further by setting
\begin{align*}
f^j&:=f^{j,+}+f^{j,-}, \\
\widehat{f^{j,+}}&:=1_{\{\xi\in \tau_j:~|\widehat{u_0}|\le C_\delta\rho_j^{-1/2}\}}\widehat{u_0},\\
\widehat{f^{j,-}}&:=1_{\{\xi\in -\tau_j:~|\widehat{u_0}|\le C_\delta\rho_j^{-1/2}\}}\widehat{u_0}.
\end{align*}
Since $u_0$ is real, $\overline{\widehat{u_0}}(\xi)=\widehat{u_0}(-\xi)$, which yields that $$\overline{\widehat{f^{j,+}}}(\xi)=\widehat{f^{j,-}}(-\xi), \text{ and }f^{j,-}=\overline{f^{j,+}}.$$
Hence $$f^j=2\re f^{j,+}.$$
Now we return to prove Theorem \ref{thm:Airy-prof-real}. We repeat the process above for each real valued $u_n$ to obtain $v^1_n,\ldots, v^N_n$ and real-valued $e_n^N$ such that
\begin{equation}\label{eq:decmp-real-1}
u_n=\sum_{j=1}^N 2\re (v_n^j) +e_n^N,
\end{equation}
and \begin{equation}\label{eq:decmp-real-2}
\sqrt{\rho_n^j}|\widehat{v_n^j}(\rho_n^j\xi+\xi_n^j)|\le C_\delta 1_K(\xi), \text{ with } \xi_n^j>0,
\text{ for some compact } K,
\end{equation}
\begin{equation}\label{eq:decmp-real-3}
\|u_n\|^2_{L^2}=\sum_{j=1}^N4\|\re (v_n^j)\|^2_{L^2}+\|e_n^N\|^2_{L^2}.
\end{equation}
Still we define the real version of the orthogonality condition on the sequence $(\rho_n^j,\xi_n^j)_{n\ge 1}\in (0,+\infty)^2$ as before: for $j\neq k$,
\begin{equation}\label{eq:decmp-real-4}\lim_{n\to \infty}\left(\dfrac
{\rho_n^j}{\rho_n^k}+\dfrac {\rho_n^k}{\rho_n^j}+\dfrac {|\xi_n^j-\xi_n^k|}
{\rho_n^j}\right)=\infty.\end{equation}
Based on \eqref{eq:decmp-real-1} and \eqref{eq:decmp-real-2}, the basic idea of obtaining the real version is to apply the procedure in the previous section to $v_n^j$, and then take the real part. The only issue here is to show that the error term is still small in the Strichartz norm, and the almost orthogonality in $L^2$ norm still holds. We omit the details and state the following
\begin{lemma}[Real version: extraction of $\rho_n^j$ and $\xi_n^j$]\label{le:scales-cores-real}
Let $(u_n)_{n\ge 1}$ be a sequence of real-valued functions with $\|u_n\|_{L^2}\le 1$. Then up to a subsequence, for any $\delta>0$, there exists $N=N(\delta)$, an orthogonal family $(\rho_n^j, \xi_n^j)_{1\le j\le N \atop n\ge 1}\in (0,\infty)^2$ satisfying \eqref{eq:decmp-real-4} and a sequence $(f_n^j)_{1\le j\le N\atop n\ge 1}\in L^2$ such that, for every $1\le j\le N$, there is a compact set K in $\R$  such that
\begin{equation}\label{eq:cmp-suppt-real}
\sqrt{\rho_n^j} |\widehat{f_n^j}(\rho_n^j\xi+\xi_n^j)|\le C_\delta 1_K(\xi),
\end{equation}
and for any $N\ge 1$, there exists a real valued $q_n^N\in L^2$ such that
\begin{equation}
\label{eq:prof-0-real}u_n=2\sum_{j=1}^{N}\re(f^j_n)+q_n^N,
\end{equation}with
\begin{equation}\label{eq:err-0-real}
\|D^{\frac 16}e^{-t\partial_x^3}q_n^N\|_{L^6_{t,x}}\le \delta,
\end{equation} and for any $N\ge 1$,
\begin{equation}\label{eq:almost-ortho-0-real}
\lim_{n\to\infty}\left(\|u_n\|^2_{L^2}-\left(\sum_{j=1}^{N}4\|\re(f^j_n)\|^2_{L^2}+\|q_n^N\|^2_{L^2}\right)
\right)=0.
\end{equation}
\end{lemma}
Then we focus on decomposing $f_n^j$ further as in Lemma \ref{le:spa-time}. Taking real parts automatically produces a decomposition for $\re(f_n^j)$. We will be sketchy on how to resolve issues of the convergence of the error term and the almost $L^2$ orthogonality.
\begin{lemma}[Real version:~extraction of $x_n^{j,\alpha}$ and $s_n^{j,\alpha}$]\label{le:spa-time-real}
Let $(f_n)_{n\ge 1}$ be a sequence of real-valued functions and $\|f_n\|_{L^2}\le 1$ satisfying
\begin{equation*}
\sqrt{\rho_n}|\widehat{f_n}(\rho_n(\xi+(\rho_n)^{-1}\xi_n))|\le F(\xi)
\end{equation*} with $F\in L^{\infty}(K)$ for some compact set $K$ and $\xi_n>0$.
Then up to a subsequence, there exists a family $(y_n^{\alpha}, s_n^{\alpha})\in \R\times\R$ and a sequence of complex-valued functions $(\phi^{\alpha})_{\alpha\ge 1}\in L^2$ such that,
if $\alpha\neq \beta$,
\begin{equation}\label{eq:ortho-s-t-real}
\lim_{n\to\infty}\left(\left|y_n^{\beta}-y_n^{\alpha}+\dfrac {3(s_n^{\beta}-s_n^{\alpha})(\xi_n)^2}{(\rho_n)^2}\right|+\left|\dfrac
{3(s_n^{\beta}-s_n^{\alpha})\xi_n}{\rho_n}\right|+\left|s_n^{\beta}-s^{\alpha}_n\right|\right)=\infty,
\end{equation}
and for each $A\ge 1$, there exists $e_n^A\in L^2$ of complex value such that
\begin{equation}\label{eq:prof-s-t-real}
f_n(x)=\sum_{\alpha=1}^{A}\widetilde{g_n^\alpha}(\phi^\alpha)(x)+\re(e_n^A)(x),
\end{equation} where
$$\widetilde{g_n^\alpha}(\phi^\alpha)(x)=\sqrt{\rho_n}e^{s^{\alpha}_n\partial^3_x}[\re(e^{i(\cdot)\rho_n^{-1}\xi_n} \phi^{\alpha})](\rho_nx-y_n^{\alpha}),$$ with $\xi_n^j\equiv 0$ when $\rho^{-1}_n\xi_n$ converges to some finite limit, and \begin{equation}\label{eq:err-s-t-real}
\lim_{A\to \infty}\lim_{n\to \infty}\|D^{\frac 16}e^{-t\partial_x^3}\re(e_n^A)\|_{L^6_{t,x}}=0, \end{equation}
and for any $A\ge 1$,
\begin{equation}\label{eq:almost-ortho-s-t-real}
\lim_{n\to\infty}\left(\|f_n\|^2_{L^2}-\left(\sum_{\alpha=1} ^{A}\|\re(e^{i(\cdot)\rho_n^{-1}\xi_n} \phi^{\alpha})\|^2_{L^2}+\|\re (e^A_n)\|^2_{L^2}\right)\right)=0.
\end{equation}
Moreover, for any $\alpha\neq\beta$,
   \begin{equation}\label{eq:s-t-real-1}
   \lim_{n\to\infty}\left|\langle\widetilde{g_n^\alpha}(\phi^{\alpha}),
   \widetilde{g_n^\beta}(\phi^{\beta})\rangle_{L^2}\right|=0,
    \end{equation}
and for any $1\le \alpha\le A$, \begin{equation}\label{eq:s-t-real-2}
    \lim_{n\to\infty}\left|\langle\widetilde{g_n^\alpha}(\phi^{\alpha}),\re(e_n^A)\rangle_{L^2}\right|=0.
    \end{equation}
\end{lemma}
\begin{proof} We briefly describe how to obtain these identities.
Equations \eqref{eq:ortho-s-t-real}, \eqref{eq:prof-s-t-real} follow along similar lines as in Lemma \ref{le:spa-time}. Equation \eqref{eq:err-s-t-real} follows from \eqref{eq:err-2} and the following point-wise inequality
$$|D^{\frac 16}e^{-t\partial_x^3}\re(e_n^A)(x)|=|\re(D^{\frac 16}e^{-t\partial_x^3}e_n^A)(x)|\le
|D^{\frac 16}e^{-t\partial_x^3}e_n^A(x)|.$$
Equation \eqref{eq:almost-ortho-s-t-real} follows from \eqref{eq:s-t-real-1} and \eqref{eq:s-t-real-2}, which are proven similarly as in Corollary \ref{coro:strong-decoupling}.
\end{proof}

\section{Final decomposition:~proof of Theorems \ref{thm:Airy-prof} and \ref{thm:Airy-prof-real}}\label{sec:proof-airy-prof}
In this section, we will only prove the complex version Theorem \ref{thm:Airy-prof}
by following the approach in \cite{Keraani:2001:profile-schrod-H^1}; the real version Theorem \ref{thm:Airy-prof-real} can be obtained similarly. We go back to the decompositions
\eqref{eq:prof-0}, \eqref{eq:prof-2} and set
$$(h_n^j, \xi_n^j,x_n^{j,\alpha},t_n^{j,\alpha}):=((\rho_n^j)^{-1},
\xi_n^j,(\rho_n^j)^{-1}y_n^{j,\alpha}, (\rho_n^j)^{-3}s_n^{j,\alpha}).$$
Then we use Remark \ref{re:reduction} and put all the error terms together,
\begin{equation}\label{eq:prof-4}
u_n=\sum_{1\le j\le N, \xi_n^j\equiv 0\atop \text{ or } |h_n^j\xi_n^j|\to \infty}\sum_{\alpha=1}^{A_j}
e^{t_n^{j,\alpha}\partial_x^3}g_n^{j,\alpha}[e^{i(\cdot)h_n^j\xi_n^j}\phi^{j,\alpha}]+w_n^{N, A_1,\ldots, A_N},
\end{equation} where $g_n^{j,\alpha}=g_{0,x_n^{j,\alpha}, h_n^j}\in G$ and
\begin{equation}\label{eq:err-3}
 w_n^{N, A_1,\ldots, A_N}=\sum_{j=1}^N e_n^{j, A_j}+q_n^N.
\end{equation}
We enumerate the pairs $(j,\alpha)$ by $\omega$ satisfying
\begin{equation}\label{eq:re-label}
 \omega(j,\alpha)<\omega(k,\beta) \text{ if } j+\alpha<k+\beta \text{ or }j+\alpha=k+\beta \text{ and
 }j<k.
\end{equation}  After re-labeling, Equation \eqref{eq:prof-4} can be further rewritten as
 \begin{equation}\label{eq:prof-7}
u_n=\sum_{1\le j\le l, \xi_n^j\equiv 0\atop \text{ or } |h_n^j\xi_n^j|\to \infty}
e^{t_n^j\partial_x^3}g_n^j[e^{i(\cdot)h_n^j\xi_n^j}\phi^j]+w_n^l,
\end{equation}  where $w_n^l=w_n^{N, A_1,\ldots, A_N}$ with $l=\sum_{j=1}^N A_j$. To establish Theorem \ref{thm:Airy-prof}, we are thus left with three points to investigate.

$\mathbf{\emph{1.}}$ The family $\Gamma_n^j=(h_n^j,\xi_n^j,t_n^j,x_n^j)$ is pairwise orthogonal, i.e., satisfying Definition \ref{def-ortho}. In fact, we have two possibilities:
\begin{itemize}
\item The two pairs are in the form $\Gamma_n^j=(h_n^i, \xi_n^i, t_n^{i,\alpha}, x_n^{i,\alpha})$ and
$\Gamma_n^k=(h_n^m,\xi_n^m,t_n^{m,\beta},x_n^{m,\beta})$ with $i\neq m$. In
this case, the orthogonality follows from that \begin{equation*}\lim_{n\to \infty} \left(\frac
{h_n^i}{h_n^m}+\frac {h_n^m}{h_n^i}+h_n^i|\xi_n^i-\xi_n^m|\right)=\infty, \end{equation*} which is \eqref{eq:ortho-0} in Lemma \ref{le:scale-core-complex}.

\item The two pairs are in form $\Gamma_n^j=(h_n^i, \xi_n^i,t_n^{i,\alpha}, x_n^{i,\alpha})$ and
$\Gamma_n^k=(h_n^i,\xi_n^i,t_n^{i,\beta},x_n^{i,\beta})$ with $\alpha\neq \beta$. In this case, the orthogonality follows from
$$\lim_{n\to \infty}\left(\frac {|t_n^{i,\beta}-t_n^{i, \alpha}|}{(h_n^i)^3}+\frac {3|t_n^{i,\beta}-
t_n^{i,\alpha}||\xi_n^i|}{(h_n^i)^2}+\left|\frac
{x_n^{i,\beta}-x_n^{i,\alpha}+3(t_n^{i,\beta}-t_n^{i,\alpha})(\xi_n^i)^2}{h_n^i}\right|\right)=\infty,$$
which is \eqref{eq:ortho-2} in Lemma \ref{le:spa-time}.
\end{itemize}
$\mathbf{\emph{2.}}$ The almost orthogonality identity \eqref{eq:almost-ortho} is satisfied. In fact, combining \eqref{eq:almost-ortho-0} and \eqref{eq:almost-ortho-2}, we obtain that
for any $N\ge 1$,
\begin{align*}
 \|u_n\|^2_{L^2} &=\sum_{j=1}^{N}\left(\sum_{\alpha=1}^{A_j} \|\phi^{j,\alpha}\|^2_{L^2}
 + \|e_n^{j,A_j}\|^2_{L^2}\right)+ \|q_n^N\|^2_{L^2}+o_n(1)\\
 &=\sum_{j=1}^{N}\left(\sum_{\alpha=1}^{A_j} \|\phi^{j,\alpha}\|^2_{L^2}
 \right)+ \|w_n^{N,A_1,\ldots,A_N}\|^2_{L^2}+o_n(1)\\
 &=\sum_{j=1}^{l}\|\phi^j\|^2_{L^2}+ \|w_n^l\|^2_{L^2}+o_n(1),\\
\end{align*} where $\lim_{n\to\infty}o_n(1)=0$. Note that we have used the fact that $$\|w_n^l\|^2_{L^2}=\|w_n^{N,A_1,\ldots,A_N}\|^2_{L^2}=\sum_{j=1}^N\|e_n^{j,A_j}\|^2_{L^2}+\|q_n^N\|^2_{L^2},$$
which is due to the disjoint supports on the Fourier side.

$\mathbf{\emph{3.}}$ The remainder $e^{-t\partial_x^3}\omega_n^{N, A_1, \ldots, A_N}$ converges to zero in the Strichartz norm. In view of the adapted enumeration, we have to
prove that
\begin{equation}\label{eq:err-4}
 \lim_{n\to \infty} \|D^{1/6}e^{-t\partial_x^3}\omega_n^{N, A_1, \ldots, A_N}\|_{L^6_{t,x}}
 \to 0, \text{ as } \inf_{1\le j\le N}\{N, j+A_j\}\to \infty.
\end{equation}
Let $\delta>0$ be an arbitrarily small number. Take $N_0$ such that, for every $N\ge N_0$,
\begin{equation}\label{eq:err-4-1}
\lim_{n\to \infty}\|D^{1/6} e^{-t\partial_x^3} q_n^N\|_{L^6_{t,x}}\le \delta/3.
\end{equation}
For every $N\ge N_0$, there exists $B_N$ such that, whenever $A_j\ge B_N$,
\begin{equation}\label{eq:err-4-2}
\lim_{n\to \infty}\|D^{1/6} e^{-t\partial_x^3} e_n^{j, A_j}\|_{L^6_{t,x}}\le
\delta/3N.
\end{equation}
The remainder $w_n^{N,A_1,\ldots,A_N}$ can be rewritten in the form
\begin{equation*}
  w_n^{N,A_1,\ldots,A_N}=q_n^N+\sum_{1\le j\le N} w_n^{j, A_j\vee B_N}+S_n^{N, A_1,\ldots, A_N},
\end{equation*} where $A_j\vee B_N:=\max\{A_j,B_N\}$ and
\begin{equation*}
S_n^{N, A_1,\ldots, A_N}=\sum_{1\le j\le N\atop A_j<B_N} (w_n^{j, A_j}-w_n^{j, B_N}),
\end{equation*}that is,
\begin{equation*}
S_n^{N, A_1,\ldots, A_N}=\sum_{1\le j\le N\atop A_j<B_N} \sum_{A_j<\alpha\le
B_N}e^{t_n^{j,\alpha}\partial_x^3}g_n^{j,\alpha}[e^{i(\cdot) h_n^j\xi_n^j}\phi^{j,\alpha}]
\end{equation*}  with $\xi_n^j\equiv 0$ when $\lim_{n\to \infty}|h_n^j\xi_n^j|<\infty$.
From \eqref{eq:err-4-1} and \eqref{eq:err-4-2}, it follows that
\begin{equation}\label{eq:err-4-3}
\lim_{n\to \infty}\|D^{1/6} e^{-t\partial_x^3} w_n^{N,A_1,\ldots,A_N}\|_{L^6_{t,x}}\le 2\delta/3+\lim_{n\to \infty}\|D^{1/6} e^{-t\partial_x^3} S_n^{N,
A_1,\ldots,A_N}\|_{L^6_{t,x}}.
\end{equation}
Now we need the following almost-orthogonality result
\begin{lemma}\label{le:weak-converg-2}
Let $\Gamma_n^j=(h_n^j, \xi_n^j, x_n^j, t_n^j)$ be a family of orthogonal sequences. Then for every
$l\ge 1$,
\begin{equation}\label{eq:weak-converg-2}
\lim_{n\to \infty}\left(\|\sum_{j=1}^l D^{1/6}e^{-(t-t_n^j)\partial_x^3}g_n^j[e^{i(\cdot)h_n^j\xi_n^j}\phi^j]\|^6_{L^6_{t,x}}-
\sum_{j=1}^l \|D^{1/6}e^{-(t-t_n^j)\partial_x^3}g_n^j[e^{i(\cdot)h_n^j\xi_n^j}\phi^j]\|^6_{L^6_{t,x}}\right)=0,
\end{equation}with $\xi_n^j\equiv 0$ when $\lim_{n\to \infty}|h_n^j\xi_n^j|<\infty$.
\end{lemma}
Suppose this lemma were proven, we show how to conclude the proof of \eqref{eq:err-4}. From Lemma
\ref{le:weak-converg-2}, it follows that
\begin{equation}\label{eq:err-4-4}
\lim_{n\to \infty}\|D^{1/6} e^{-t\partial_x^3} S_n^{N, A_1,\ldots,A_N}\|^6_{L^6_{t,x}} =\sum_{1\le j\le N\atop A_j<B_N} \sum_{A_j<\alpha\le B_N}\lim_{n\to \infty} \|D^{1/6}
e^{-(t-t_n^{j,\alpha})\partial_x^3}g_n^{j,\alpha}[e^{i(\cdot) h_n^j\xi_n^j}\phi^{j,\alpha}]\|^6_{L^6_{t,x}}.
\end{equation}
The Strichartz inequality gives that
\begin{equation}\label{eq:err-4-5}
\sum_{1\le j\le N\atop A_j<B_N} \sum_{A_j<\alpha\le
B_N}\|D^{1/6}e^{-(t-t_n^{j,\alpha})\partial_x^3}g_n^{j,\alpha}[e^{i(\cdot) h_n^j\xi_n^j}\phi^{j,\alpha}]\|^6_{L^6_{t,x}}\lesssim \sum_{1\le j\le N\atop A_j<B_N} \sum_{A_j<\alpha\le B_N}\|\phi^{j,\alpha}\|^6_{L^2}\le \sum_{j,\alpha}\|\phi^{j,\alpha}\|^6_{L^2}.
\end{equation}
On the other hand, $\sum_{j,\alpha}\|\phi^{j,\alpha}\|^2_{L^2}$ is convergent; hence the right-hand side of \eqref{eq:err-4-5} is finite. This shows
\begin{equation}\label{eq:err-4-6}
\left(\sum_{j,\alpha\atop\alpha>A_j}\|D^{1/6}e^{-(t-t_n^{j,\alpha})\partial_x^3}
g_n^{j,\alpha}[e^{i(\cdot) h_n^j\xi_n^j}\phi^{j,\alpha}]\|^6_{L^6_{t,x}}\right)^{1/6}\le \delta/3
\end{equation} provided that $\inf_{1\le j\le N}\{N,j+A_j\}$ is large enough. Combining
\eqref{eq:err-4-3}, \eqref{eq:err-4-4} and \eqref{eq:err-4-6}, we obtain
\begin{equation}\label{eq:err-4-7}
\lim_{n\to \infty}\|D^{1/6} e^{-t\partial_x^3} w_n^{N,A_1,\ldots,A_N}\|_{L^6_{t,x}}=0
\end{equation} provided that $\inf_{1\le j\le N}\{N,j+A_j\}$ is large enough. Hence the proof of
\eqref{eq:err-4} is complete.

\begin{proof}[Proof of Lemma \ref{le:weak-converg-2}] By using the H\"older inequality, we need to show that for $j\neq k$, as $n$ goes to infinity,
\begin{equation}\label{eq:weak-converg-2-2}
\|D^{1/6}e^{-(t-t_n^j)\partial_x^3}g_n^j[e^{i(\cdot) h_n^j\xi_n^j}\phi^j]D^{1/6}e^{-(t-t_n^k)\partial_x^3}
g_n^k[e^{i(\cdot) h_n^k\xi_n^k}\phi^k]\|_{L^3_{t,x}}\to 0.
\end{equation}
By the pigeonhole principle, we can assume that $\xi_n^j$ and $\xi_n^k$ are of the same sign if they are not zero; moreover by a density argument, we also assume that $\phi^j$ and $\phi^k$ are Schwartz functions with compact Fourier supports. Evidence in favor of \eqref{eq:weak-converg-2-2} is that, if $\lim_{n\to \infty}|h_n\xi_n|=\infty$, $D^{1/6}e^{-(t-t_n)\partial_x^3}g_n[e^{i(\cdot) h_n\xi_n}\phi]$ is somehow a Schr\"odinger wave in the sense of Remark \ref{re:airy-schr}. For the pairwise orthogonal Schr\"odinger waves, however, the analogous result to \eqref{eq:weak-converg-2-2} is true, see e.g., \cite{Merle-Vega:1998:profile-schrod}, \cite{Carles-Keraani:2007:profile-schrod-1d} and \cite{Begout-Vargas:2007:profile-schrod-higher-d}.

To prove \eqref{eq:weak-converg-2-2} we will have two possibilities.
First, the two pairs are in the form $\Gamma_n^j=(h_n^i, \xi_n^i, t_n^{i,\alpha}, x_n^{i,\alpha})$ and
$\Gamma_n^k=(h_n^m,\xi_n^m,t_n^{m,\beta},x_n^{m,\beta})$ with $i\neq m$. In
this case, the orthogonality is given by
\begin{equation*}\lim_{n\to \infty} \left(\frac{h_n^i}{h_n^m}+\frac{h_n^m}{h_n^i}+h_n^i|\xi_n^i-\xi_n^m|\right)=\infty.\end{equation*}
So we have two subcases. We begin with the case where $\lim_{n\to\infty}h_n^i|\xi_n^i-\xi_n^m|=\infty$; moreover, we may assume that $h_n^i=h_n^m$ for all $n$ (when both limits are infinity, it can be done similarly by using the argument below). By changing variables, the left hand side of \eqref{eq:weak-converg-2-2} equals
\begin{equation}\label{eq:weak-converg-2-3}
\left\|D^{1/6}e^{-t\partial_x^3}\left(e^{i(\cdot)h_n^i\xi_n^i}\phi^{i,\alpha}\right)
D^{1/6}e^{-(t+\frac {t_n^{i,\alpha}-t_n^{m,\beta}}{(h_n^i)^3})\partial_x^3}
\left(e^{i(\cdot)h_n^i\xi_n^m}\phi^{m,\beta}\right)(x+\frac{x_n^{m,\alpha}-x_n^{i,\beta}}{h_n^i})\right\|_{L^3_{t,x}}.
\end{equation}
The integrand above equals
\begin{align*}
\int\int &e^{ix[(\xi+h_n^i\xi_n^i)+(\eta+h_n^i\xi_n^m)]+it[(\xi+h_n^i\xi_n^i)^3
+(\eta+h_n^i\xi_n^m)^3]} |\xi+h_n^i\xi_n^i|^{1/6}|\eta+h_n^i\xi_n^m|^{1/6} \times\\
&\times e^{i(\eta+h_n^i\xi_n^m)\frac {x_n^{i,\alpha}-x_n^{m,\beta}}{h_n^i}+i(\eta+h_n^i\xi_n^m)^3\frac {t_n^{i,\alpha}-t_n^{m,\beta}}{(h_n^i)^3}} \widehat{\phi^{i,\alpha}}(\xi)\widehat{\phi^{m,\beta}}(\eta)d\xi d\eta.
\end{align*}
Changing variables again $a:=(\xi+h_n^i\xi_n^i)+(\eta+h_n^i\xi_n^m)$ and $b:=(\xi+h_n^i\xi_n^i)^3
+(\eta+h_n^i\xi_n^m)^3$ followed by the Hausdorff-Young inequality, we see that \eqref{eq:weak-converg-2-3} is bounded by
\begin{equation*}
C\left(\int\int\frac {|\xi+h_n^i\xi_n^i|^{1/4}|\eta+h_n^i\xi_n^m|^{1/4}
|\widehat{\phi^{i,\alpha}}(\xi)\widehat{\phi^{m,\beta}}(\eta)|^{3/2}}
{|\xi+h_n^i\xi_n^i+\eta+h_n^i\xi_n^m|^{1/2}|\xi-\eta+h_n^i(\xi_n^i-\xi_n^m)|^{1/2}
}d\xi d\eta\right)^{2/3}.
\end{equation*}
We consider two subcases according to the limits of $|h_n^i\xi_n^i|$ and $|h_n^m\xi_n^m|$. Note that $\lim_{n\to\infty} h_n^i|\xi_n^i-\xi_n^m|=\infty$, then either both are infinity or only one is.
\begin{itemize}
\item In the former case, since $\xi_n^i$ and $\xi_n^m$ are of the same sign, we have
$$\frac{|\xi+h_n^i\xi_n^i|^{1/4}|\eta+h_n^i\xi_n^m|^{1/4}}{|\xi+\eta+h_n^i(\xi_n^i
+\xi_n^m)|^{1/2}}\sim \frac {|\xi_n^i\xi_n^m|^{1/4}}{|\xi_n^i+\xi_n^m|^{1/2}}\lesssim 1.$$
Then \eqref{eq:weak-converg-2-3} is further bounded by $C_{\phi^{i,\alpha},\phi^{m,\beta}}(h_n^i|\xi_n^i-\xi_n^m|)^{-1/3}$, which goes to zero as $n$ goes to infinity.

\item In the latter case, say $\lim_{n\to \infty} |h_n^i\xi_n^i|=\infty$, we will have $\xi_n^m=0$. Then
    $$\frac{|\xi+h_n^i\xi_n^i|^{1/4}|\eta+h_n^i\xi_n^m|^{1/4}}{|\xi+\eta+h_n^i(\xi_n^i
+\xi_n^m)|^{1/2}}\lesssim |h_n^i\xi_n^i|^{-1/4}.$$
Then \eqref{eq:weak-converg-2-3} is further bounded by $C_{\phi^{i,\alpha},\phi^{m,\beta}}|h_n^i\xi_n^i|^{-1/2}$, which goes to zero as $n$ goes to infinity.
\end{itemize}
Under the first possibility, we still need to consider the case when $\lim_{n\to\infty}\left(\frac{h_n^i}{h_n^m}+\frac{h_n^m}{h_n^i}\right)=\infty$. We may assume that $\lim_{n\to\infty}|h_n^i\xi_n^i-h_n^m\xi_n^m|<\infty$. It follows that $\lim_{n\to\infty}|h_n^i\xi_n^i|$ and $\lim_{n\to\infty}|h_n^m\xi_n^m|$ are finite or infinite simultaneously. We will consider the case where they are both infinite since the other follows similarly. Under this consideration, we deduce that
  $$ \left|\frac {h_n^m\xi_n^m}{h_n^i\xi_n^i}\right|\sim 1$$
for sufficiently large $n$. To prove \eqref{eq:weak-converg-2-2}, we will use the idea of regarding the profile term as a Schr\"odinger wave as in Remark \ref{re:airy-schr}. We recall
\begin{align*}
&D^{1/6}e^{-(t-t_n^j)\partial_x^3}g_n^j[e^{i(\cdot) h_n^j\xi_n^j}\phi^j]
=(h_n^i)^{-1/2}|\xi_n^i|^{1/6}e^{i\xi_n^i(x-x_n^{i,\alpha})+i(\xi_n^i)^3(t-t_n^{i,\alpha})}\\
&\quad \times \int e^{i\xi[\frac {x-x_n^{i,\alpha}}{h_n^i}+3(\xi_n^i)^2\frac {t-t_n^{i,\alpha}}{h_n^i}]+i\xi^3\frac {t-t_n^{i,\alpha}}{(h_n^i)^3}+3i\xi^2\xi_n^i\frac {t-t_n^{i,\alpha}}{(h_n^i)^2}}|1+\frac {\xi}{h_n^i\xi_n^i}|^{1/6}\widehat{\phi^{i,\alpha}}d\xi,
\end{align*} Similarly for $D^{1/6}e^{-(t-t_n^k)\partial_x^3}g_n^k[e^{i(\cdot) h_n^k\xi_n^k}\phi^k]$.
For any $R>0$, we denote \begin{align*}
    A^i_R&:=\{(t,x)\in \R\times \R: \left|3\xi_n^i\frac {t-t_n^{i,\alpha}}{(h_n^i)^2}\right|+
    \left|\frac {x-x_n^{i,\alpha}}{h_n^i}+3(\xi_n^i)^2\frac {t-t_n^{i,\alpha}}{h_n^i}\right|\le R\},\\
    A^m_R&:=\{(t,x)\in \R\times \R: \left|3\xi_n^m\frac {t-t_n^{m,\beta}}{(h_n^m)^2}\right|+
    \left|\frac {x-x_n^{m,\beta}}{h_n^m}+3(\xi_n^m)^2\frac {t-t_n^{m,\beta}}{h_n^m}\right|\le R\}.
\end{align*}
By the H\"older inequality, the Strichartz inequality and Remark \ref{re:airy-schr}, we only need to show, for a large $R>0$,
\begin{equation}\label{eq:loc-3}\lim_{n\to\infty}\|D^{1/6}e^{-(t-t_n^i)\partial_x^3}
g_n^j[e^{i(\cdot) h_n^i\xi_n^i}\phi^j]D^{1/6}e^{-(t-t_n^m)\partial_x^3}
g_n^k[e^{i(\cdot) h_n^m\xi_n^m}\phi^k]\|_{L^3_{t,x}(A^i_R\cap A^m_R)}=0.
\end{equation}
Indeed, $\R^2\setminus (A_R^i\cap A_R^m)\subset (\R^2\setminus A_R^i)\cup (\R^2\setminus A_R^m)$; here we only consider the integration over the region $\R^2\setminus A_R^i$ since the other case is similar. By the H\"older inequality and the Strichartz inequality,
\begin{align*}
&\|D^{1/6}e^{-(t-t_n^i)\partial_x^3}
g_n^j[e^{i(\cdot) h_n^i\xi_n^i}\phi^j]D^{1/6}e^{-(t-t_n^m)\partial_x^3}
g_n^k[e^{i(\cdot) h_n^m\xi_n^m}\phi^k]\|_{L^3_{t,x}(\R^2\setminus A_R^i)} \\
&\lesssim \|D^{1/6}e^{-(t-t_n^i)\partial_x^3}
g_n^j[e^{i(\cdot) h_n^i\xi_n^i}\phi^j]\|_{L^6_{t,x}(\R^2\setminus A_R^i)}\|D^{1/6}e^{-(t-t_n^m)\partial_x^3}
g_n^k[e^{i(\cdot) h_n^m\xi_n^m}\phi^k]\|_{L^6_{t,x}}\\
&\lesssim \|\phi^k\|_{L^2}\|D^{1/6}e^{-(t-t_n^i)\partial_x^3}
g_n^j[e^{i(\cdot) h_n^i\xi_n^i}\phi^j]\|_{L^6_{t,x}(\R^2\setminus A_R^i)}.
\end{align*}
Let $x':=\frac {x-x_n^{i,\alpha}+3(\xi_n^i)^2(t-t_n^{i,\alpha})}{h_n^i}$ and $t':=\frac {3\xi_n^i(t-t_n^{i,\alpha})}{(h_n^i)^2}$. Then a change of variables and similar computations as in Remark \ref{re:airy-schr} show that
\begin{align*}
&\|D^{1/6}e^{-(t-t_n^i)\partial_x^3}
g_n^j[e^{i(\cdot) h_n^i\xi_n^i}\phi^j]\|_{L^6_{t,x}(\R^2\setminus A_R^i)}\\
&\lesssim \|\int e^{i(x'\xi+t'\xi^2)+i\frac {\xi^3t'}{3h_n^i\xi_n^i}}|1+\frac {\xi}{h_n^i\xi_n^i}|^{1/6}\widehat{\phi^{i,\alpha}}(\xi)d\xi\|_{L^6_{t',x'}(|t'|+|x'|\ge R)}\\
&\to \|e^{-it'\Delta}\phi^{i,\alpha}\|_{L^6_{t',x'}(|t'|+|x'|\ge R)}\to 0,
\end{align*} as $n\to \infty$ followed by $R\to \infty$.
Returning to \eqref{eq:loc-3}, if using $L^{\infty}$-bounds for the integrands, we see that it is bounded by
\begin{align*}
&C\|D^{1/6}e^{-(t-t_n^i)\partial_x^3}g_n^j[e^{i(\cdot)
h_n^i\xi_n^i}\phi^j]\|_{L^\infty}\|D^{1/6}e^{-(t-t_n^m)\partial_x^3}
g_n^k[e^{i(\cdot) h_n^m\xi_n^m}\phi^k]\|_{L^\infty}\min\{|A_R^i|^{1/3}, |A_R^m|^{1/3}\}\\
&\le C_{R,\phi^j,\phi^k}(h_n^ih_n^m)^{-1/2}|\xi_n^i\xi_n^m|^{1/6}\min\{[(h_n^i)^3|\xi_n^i|^{-1}]^{1/3},\,
[(h_n^m)^3|\xi_n^m|^{-1}]^{1/3}\}\\
&\le C_{R,\phi^j,\phi^k} \min\{(\frac {h_n^i}{h_n^m})^{2/3}\left|\frac {h_n^m\xi_n^m}{h_n^i\xi_n^i}\right|^{1/6},\quad
(\frac {h_n^m}{h_n^i})^{2/3}\left|\frac {h_n^i\xi_n^i}{h_n^m\xi_n^m}\right|^{1/6}\}.
\end{align*}
Hence when $\lim_{n\to\infty}\left(\frac{h_n^i}{h_n^m}+\frac{h_n^m}{h_n^i}\right)=\infty$, \eqref{eq:weak-converg-2-2} holds.

Secondly, the two pairs are in form $\Gamma_n^j=(h_n^i, \xi_n^i,t_n^{i,\alpha}, x_n^{i,\alpha})$ and
$\Gamma_n^k=(h_n^i,\xi_n^i,t_n^{i,\beta},x_n^{i,\beta})$ with $\alpha\neq \beta$. In this case, the orthogonality is given by
$$\lim_{n\to \infty}\left(\frac {|t_n^{i,\beta}-t_n^{i, \alpha}|}{(h_n^i)^3}+\frac {3|t_n^{i,\beta}-
t_n^{i,\alpha}||\xi_n^i|}{(h_n^i)^2}+\frac{|x_n^{i,\beta}-x_n^{i,\alpha}+3(t_n^{i,\beta}-t_n^{i,\alpha})
(\xi_n^i)^2|}{h_n^i}\right)=\infty.$$
We assume $\lim_{n\to\infty}|h_n^i\xi_n^i|=\infty$ since the other case is similar. We expand the left-hand side of \eqref{eq:weak-converg-2-2} out, which is equal to
\begin{align*}
&(h_n^i)^{-\frac 43}\|D^{1/6}e^{-\frac {t-t_n^{i,\alpha}}{(h_n^i)^3}\partial_x^3}[e^{i(\cdot)
 h_n^i\xi_n^i}\phi^{i,\alpha}](\frac {x-x_n^{i,\alpha}}{h_n^i})\,D^{1/6}e^{-\frac {t-t_n^{m,\beta}}{(h_n^i)^3}\partial_x^3}
 [e^{i(\cdot)h_n^i\xi_n^i}\phi^{m,\beta}](\frac {x-x_n^{m,\beta}}{h_n^i})\|_{L^3_{t,x}}\\
 &=\frac {|\xi_n^{i}|^{1/3}}{h_n^i}\|\int e^{i[\frac{\eta(x-x_n^{i,\alpha}+3(t-t_n^{i,\alpha})
 (\xi_n^i)^2)}{h_n^i}+\frac {\eta^3(t-t_n^{i,\alpha})}{(h_n^i)^3}+
 \frac {3\eta^2(t-t_n^{i,\alpha})\xi_n^i}{(h_n^i)^2}]}|1+\frac {\eta}{h_n^i\xi_n^i}|^{1/6}
 \widehat{\phi^{i,\alpha}}(\eta)d\eta\times \nonumber\\
 &\times \int e^{i[\frac{\eta(x-x_n^{i,\beta}+3(t-t_n^{i,\beta})
 (\xi_n^i)^2)}{h_n^i}+\frac {\eta^3(t-t_n^{i,\beta})}{(h_n^i)^3}+
 \frac {3\eta^2(t-t_n^{i,\beta})\xi_n^i}{(h_n^i)^2}}|1+\frac {\eta}{h_n^i\xi_n^i}|^{1/6}
 \widehat{\phi^{i,\beta}}(\eta)d\eta \|_{L^3_{t,x}}
\end{align*} If changing variables
$t'=\frac {3(t-t_n^{i,\beta})\xi_n^i}{(h_n^i)^2}$ and
$x'=\frac{x-x_n^{i,\beta}+3(t-t_n^{i,\beta})
 (\xi_n^i)^2}{h_n^i}$, it reduces to
\begin{align*}
 &C\|\int e^{i\eta[x'+\frac {x_n^{i,\beta}-x_n^{i,\alpha}+ 3(t_n^{i,\beta}-t_n^{i,\alpha})
 (\xi_n^i)^2}{h_n^i}]+i\eta^3[\frac {t_n^{i,\beta}-t_n^{i,\alpha}}{(h_n^i)^3}+
 \frac {t'}{3h_n^i\xi_n^i}]+i\eta^2[t'+\frac {3(t_n^{i,\beta}-t_n^{i,\alpha})\xi_n^i}
 {(h_n^i)^2}]}\times\\
  &\times |1+\frac {\eta}{h_n^i\xi_n^i}|^{1/6}\widehat{\phi^{i,\alpha}}(\eta)d\eta
\int e^{ix'\eta+it'\eta^2}e^{i\eta^3\frac {t'}{3h_n^i\xi_n^i}}
 |1+\frac {\eta}{h_n^i\xi_n^i}|^{1/6}\widehat{\phi^{i,\beta}}(\eta)d\eta\|_{L^3_{t',x'}}
\end{align*} Then the H\"older inequality followed by the principle of the stationary phase or integration by parts, we see that \eqref{eq:weak-converg-2-2} holds.
\end{proof}
Similarly, we can obtain the following generalization of Corollary \ref{coro:strong-decoupling} about the orthogonality of profiles in $L^2$ space. Its proof will be omitted.
\begin{lemma}\label{le:weak-converg}
Assume $\Gamma_n^j=(h_n^j,\xi_n^j,t_n^j,x_n^j)$ and $\Gamma_n^k=(h_n^k,\xi_n^k,t_n^k,x_n^k)$
are pairwise orthogonal, then
\begin{equation}\label{eq:weak-converg-1} \lim_{n\to \infty}\langle
e^{t_n^j\partial_x^3}g_n^j[e^{i(\cdot)h_n^j\xi_n^j}\phi^j],
e^{t_n^k\partial_x^3}g_n^k[e^{i(\cdot)h_n^k\xi_n^k}\phi^k]
\rangle_{L^2}=0,
\end{equation}
and for $1\le j\le l$,
\begin{equation}\label{eq:weak-converg-2} \lim_{n\to \infty}\langle
e^{t_n^j\partial_x^3}g_n^j[e^{i(\cdot)h_n^j\xi_n^j}\phi^j],w_n^l\rangle_{L^2}=0,
\end{equation} with $\xi_n^j\equiv 0$ when $\lim_{n\to \infty}|h_n^j\xi_n^j|<\infty$.
\end{lemma}

\section{The existence of maximizers for the symmetric Airy Strichartz inequality}\label{sec:airy-maxi}
This section is devoted to establishing Theorem \ref{thm:airy-max}, a dichotomy result on the existence of maximizers for the symmetric Airy Strichartz inequality. First, we will exploit the idea of asymptotically embedding a Schr\"odinger solution into an approximate Airy solution. We will show that the best constant for the Airy Schr\"odinger Strichartz bounds that for the symmetric Schr\"odinger Strichartz inequality up to a constant. We will follow the approach in \cite{Tao:2007:scattering-quartic-gKdV}, in which Tao shows that any qualitative scattering result on the mass critical gKdV equation $\partial_t u+\partial_x^3 u\pm|u|^4\partial_x u=0$ automatically implies an analogous scattering result for the mass critical nonlinear Schr\"odinger equation $i\partial_t u+\partial_x^2 u\pm |u|^4u=0$.
\begin{lemma}[Asymptotic embedding of Schr\"odinger into Airy]\label{le:embedding-schr-airy}
Corresponding to Theorems \ref{thm:Airy-prof} and \ref{thm:Airy-prof-real}, respectively,
\begin{align}
\label{eq:esa-1}  S_{schr}^\C&\le 3^{1/6}S_{airy}^\C, \\
\label{eq:esa-2}  S_{schr}^\C&\le 2^{1/2}3^{1/6}S_{airy}^\R.
\end{align}
\end{lemma}
\begin{proof}We first prove \eqref{eq:esa-2}. Let $u_0$ to a maximizer to \eqref{eq:schr-max}. Since d=1, from the work in \cite{Foschi:2007:maxi-strichartz-2d}, we can assume that $u_0$ is a standard Gaussian; hence it is even and its Fourier transform is another Gaussian. Denote
$$u_N(0,x):=\frac 1{(3N)^{1/4}}\re\left(e^{ixN}u_0(\frac {x}{\sqrt{3N}})\right).$$
Let $u_N(t,x)$ solve the Airy equation \eqref{eq:airy} with initial data $u_N(0,x)$.
From the Airy Strichartz inequality,
\begin{equation}\label{eq:embedding-schr-airy-1} \|D^{1/6}u_N\|_{L^6_{t,x}}\le S_{airy}^\R\|u_N(0,x)\|_{L^2}.
\end{equation}
On the one hand, a computation shows that
\begin{equation}\label{eq:embedding-schr-airy-2}\|u_N(0,x)\|^2_{L^2}=\frac {1}{2}\int
|u_0(x)|^2+\re\left(e^{2\sqrt{3}iN^{3/2}x} u_0^2(x)\right)dx. \end{equation}
From the Riemann-Lebesgue lemma, we know the second term above rapidly goes to zero as $N\to
\infty$. On the other hand,
$$\widehat{u_N}(0,\xi)=\frac{(3N)^{1/4}}{2}\left(\widehat{u_0}(\sqrt{3N}(\xi-N))+\widehat{u_0}
(\sqrt{3N}(\xi+N))\right),$$
which yields
\begin{align*}
D^{1/6}u_N(t,x)&=\int e^{ix\xi+it\xi^3}|\xi|^{1/6}\widehat{u_N}(0,\xi)d\xi \\
&=\frac{(3N)^{1/4}}{2}\int e^{ix\xi+it\xi^3}|\xi|^{1/6}
\left(\widehat{u_0}(\sqrt{3N}(\xi-N))+\widehat{u_0}(\sqrt{3N}(\xi+N))\right)d\xi  \\
&= 2^{-1}3^{-1/4} N^{-1/{12}}e^{ixN+itN^3}\int e^{i[\eta((3N)^{-1/2}x+\sqrt{3}N^{3/2}t)
 +t\eta^2+t(3N)^{-3/2}\eta^3]}\times \\
 &\qquad\times |1+\frac {\eta}{N\sqrt{3N}}|^{1/6}
 \left(\widehat{u_0}(\eta)+\widehat{u_0}(\eta+2N\sqrt{3N})\right)d\eta.
\end{align*}
Changing variables $x'=(3N)^{-1/2}x+\sqrt{3}N^{3/2}t$ and $t'=t$, we obtain
\begin{align}\label{eq:embedding-schr-airy-3}
&\|D^{1/6}u_N(t,x)\|_{L^6_{t,x}}=
2^{-1}3^{-1/6}\|\int e^{i[x'\eta +t'\eta^2+t'(3N)^{-3/2}\eta^3]}\times \nonumber\\
 &\qquad\times |1+\frac{\eta}{N\sqrt{3N}}|^{1/6}
 \left(\widehat{u_0}(\eta)+\widehat{u_0}(\eta+2N\sqrt{3N})\right)d\eta\|_{L^6_{t',x'}}
\end{align}
Comparing \eqref{eq:embedding-schr-airy-1}, \eqref{eq:embedding-schr-airy-2},
\eqref{eq:embedding-schr-airy-3} and letting $N\to \infty$, as in Remark \ref{re:airy-schr}, we
obtain,
\begin{equation}\label{eq:embedding-schr-airy-4}
 2^{-1}3^{-1/6}\|\int e^{ix'\eta+it'\eta^2}\widehat{u_0}(\eta)d\eta\|_{L^6_{t',x'}}\le 2^{-1/2}
S_{airy}^\R\|u_0\|_{L^2}.
\end{equation}
By the choice of $u_0$, we have \begin{equation*}
 2^{-1}3^{-1/6}S_{schr}^\C\le  2^{-1/2}S_{airy}^\R,
\end{equation*}
i.e., $S_{schr}^\C\le 2^{1/2}3^{1/6}S_{airy}^\R$. Hence \eqref{eq:esa-2} follows.
To show \eqref{eq:esa-1}, we choose $\phi_N(x):=\frac 1{(3N)^{1/4}}e^{ixN}u_0(\frac
{x}{\sqrt{3N}})$. Then
\begin{equation*}\|\phi_N\|_{L^2}=\|u_0\|_{L^2},\,\|e^{-it\partial_x^2}\phi_N\|_{L^6_{t,x}
(\R\times\R)} =S_{schr}^\C\|u_0\|_{L^2}.\end{equation*}
Also an easy computation shows that
\begin{equation*}
\|D^{1/6}e^{-t\partial_x^3}\phi_N\|_{L^6_{t,x}}\to
3^{-1/6}\|e^{-it\partial_x^2}u_0\|_{L^6_{t,x}},\text{ as }N\to \infty.
\end{equation*}
From the Airy Strichartz inequality, $$\|D^{1/6}e^{-t\partial_x^3}\phi_N\|_{L^6_{t,x}}\le S_{airy}^\C\|\phi_N\|_{L^2}=S_{airy}^\C\|u_0\|_{L^2},$$ we conclude that
\eqref{eq:esa-1} follows.
\end{proof}
Now we are ready to prove Theorem \ref{thm:airy-max}.
\begin{proof}[Proof of Theorem \ref{thm:airy-max}] We only prove the complex version by using Theorem
\ref{thm:Airy-prof}. For the real version, we use Theorem \ref{thm:Airy-prof-real} instead but its
proof is similar.

We choose a maximizing sequence $(u_n)_{n\ge 1}$ with $\|u_n\|_{L^2}=1$, and decompose it into the linear profiles as in Theorem \ref{thm:Airy-prof} to obtain
\begin{equation}\label{eq:airy-max-0}
 u_n=\sum_{1\le j\le l, \xi_n^j\equiv 0 \atop \text{ or } |h_n^j\xi_n^j|\to \infty} e^{t_n^j\partial_x^3}g_n^j[e^{i(\cdot)h_n^j\xi_n^j}\phi^{j}]+w_n^l.
\end{equation}
Then from the asymptotically vanishing Strichartz norm \eqref{eq:err} and the triangle inequality,
we obtain that, up to a subsequence, for any given $\eps>0$, there exists $n_0$, for all $l\ge
n_0$ and $n\ge n_0$,
\begin{equation*}
\|\sum_{j=1}^{l}D^{1/6}e^{-(t-t_n^j)\partial_x^3}g_n^j[e^{i(\cdot)h_n^j\xi_n^j}
\phi^j]\|_{L^6_{t,x}}\ge S_{airy}^\C-\eps,
\end{equation*} with $\xi_n^j\equiv 0$ when $\lim_{n\to \infty}|h_n^j\xi_n^j|<\infty$.
On the other hand, Lemma \ref{le:weak-converg-2} yields,
\begin{equation}\label{eq:airy-max-1}
\|\sum_{j=1}^lD^{1/6}e^{-(t-t_n^j)\partial_x^3}g_n^j[e^{i(\cdot)h_n^j\xi_n^j}
\phi^j]\|^6_{L^6_{t,x}}
\le\sum_{j=1}^{l}\|D^{1/6}e^{-(t-t_n^j)\partial_x^3}g_n^j[e^{i(\cdot)h_n^j\xi_n^j}
\phi^j]\|^6_{L^6_{t,x}}+o_n(1).
\end{equation}
Then up to a subsequence, there exists $n_1$ such that, for large $n\ge n_1$ and $l\ge n_1$,
\begin{equation}\label{eq:ariy-max-2}
\sum_{j=1}^{l}\|D^{1/6}e^{-(t-t_n^j)\partial_x^3}g_n^j[e^{i(\cdot)h_n^j\xi_n^j}
\phi^j]\|^6_{L^6_{t,x}}\ge (S_{airy}^\C)^6-2\eps.
\end{equation}
Choosing $j_0$ such that $D^{1/6}e^{-(t-t_n^{j_0})\partial_x^3}g_n^{j_0}[e^{i(\cdot)h_n^{j_0}\xi_n^{j_0}}
\phi^{j_0}]$ has the biggest Strichartz norm among $1\le j \le l$, we see that, by Strichartz and the almost orthogonal identity \eqref{eq:almost-ortho},
\begin{align*}
(S_{airy}^\C)^6-2\eps &\le \|D^{1/6}e^{-(t-t_n^{j_0})\partial_x^3}g_n^{j_0}[e^{i(\cdot)h_n^{j_0}\xi_n^{j_0}}
\phi^{j_0}]\|^4_{L^6_{t,x}}
\sum_{j=1}^{l} \|D^{1/6}e^{-(t-t_n^j)\partial_x^3}g_n^j[e^{i(\cdot)h_n^j\xi_n^j}
\phi^j]\|^2_{L^6_{t,x}}\\
&\le \|D^{1/6}e^{-(t-t_n^{j_0})\partial_x^3}g_n^{j_0}[e^{i(\cdot)h_n^{j_0}\xi_n^{j_0}}
\phi^{j_0}]\|^4_{L^6_{t,x}} \sum_{j=1}^l
\left(S_{airy}^\C\|\phi^j\|_{L^2}\right)^2 \\
&\le (S_{airy}^\C)^2\|D^{1/6}e^{-(t-t_n^{j_0})\partial_x^3}g_n^{j_0}[e^{i(\cdot)h_n^{j_0}\xi_n^{j_0}}
\phi^{j_0}]\|^4_{L^6_{t,x}}.
\end{align*}
This yields,
\begin{equation}\label{eq:airy-max-6}
\|D^{1/6}e^{-(t-t_n^{j_0})\partial_x^3}g_n^{j_0}[e^{i(\cdot)h_n^{j_0}\xi_n^{j_0}}
\phi^{j_0}]\|_{L^6_{t,x}}\ge
\left((S_{airy}^\C)^{-2}[(S_{airy}^\C)^6-2\eps]\right)^{1/4}\ge S_{airy}^\C-\eps.
\end{equation}
Moreover, \eqref{eq:almost-ortho} implies that there exists $J>0$, such that \begin{equation*}
\|\phi^j\|_{L^2}\le 1/100, \forall j>J.
\end{equation*}
This, together with \eqref{eq:airy-max-6} and the Strichartz inequality \begin{equation*}
\|D^{1/6}e^{-(t-t_n^{j_0})\partial_x^3}g_n^{j_0}[e^{i(\cdot)h_n^{j_0}\xi_n^{j_0}}
\phi^{j_0}]\|_{L^6_{t,x}}\le S_{airy}^\C \|\phi^{j_0}\|_{L^2},
\end{equation*}
shows that, for $\eps$ small enough, $j_0$ is between $1$ and $J$; otherwise $S_{airy}^\C/2\le S_{airy}^\C/100$, a contradiction. Hence $j_0$ does not depend on $l$, $n$ and $\eps$. So we can freely take $\eps$ to zero without changing $j_0$. Now we split into two cases:
\begin{itemize}
\item[\emph{Case I.}]When $h_n^{j_0}\xi_n^{j_0}\to \xi^{j_0}\in \R$, we can take $\xi_n^{j_0}\equiv 0$. Then $\|D^{1/6}e^{-(t-t_n^{j_0})\partial_x^3}g_n^{j_0}(\phi^{j_0})\|_{L^6_{t,x}}=
\|D^{1/6}e^{-t\partial_x^3}\phi^{j_0}\|_{L^6_{t,x}}$. Then we take $\eps\to 0$ in
\eqref{eq:airy-max-6} to obtain $$\|\phi^{j_0}\|_{L^2}=1, \,
S_{airy}^\C=\|D^{1/6}e^{-t\partial_x^3}\phi^{j_0}\|_{L^6_{t,x}}.$$ This shows that $\phi^{j_0}$ is
a maximizer for \eqref{eq:airy-max}.

\item[\emph{Case II.}]   When $|h_n^{j_0}\xi_n^{j_0}|\to \infty$, we take $n\to\infty$ in \eqref{eq:airy-max-6} and use Remark \ref{re:airy-schr},
\begin{align*} S_{airy}^\C-\eps &\le \lim_{n\to\infty}\|D^{1/6}e^{-(t-t_n^{j_0})\partial_x^3}g_n^{j_0}[e^{i(\cdot)h_n^{j_0}\xi_n^{j_0}}
\phi^{j_0}]\|_{L^6_{t,x}} \\
&=\lim_{n\to\infty}\|D^{1/6}e^{-t\partial_x^3}[e^{i(\cdot)h_n^{j_0}\xi_n^{j_0}}
\phi^{j_0}]\|_{L^6_{t,x}}\\
&=3^{-1/6}\|e^{-it\partial_x^2}\phi^{j_0}\|_{L^6_{t,x}}\le
3^{-1/6}S_{schr}^\C\|\phi^{j_0}\|_{L^2}\\
&\le S_{airy}^\C\|\phi^{j_0}\|_{L^2}.
\end{align*}
Taking $\eps \to 0$ forces all the inequality signs to be equal. Hence we obtain $$\|\phi^{j_0}\|_{L^2}=1,\,S_{airy}^\C=3^{-1/6}S_{schr}^\C$$ and
$S_{airy}^\C=\lim_{n\to\infty}\|D^{1/6}e^{-t\partial_x^3}[e^{i(\cdot)h_n^{j_0}\xi_n^{j_0}}
\phi^{j_0}]\|_{L^6_{t,x}}=3^{-1/6}\|e^{-it\partial_x^2}\phi^{j_0}\|_{L^6_{t,x}}$. This shows that
$S_{schr}^\C=\|e^{-it\partial_x^2}\phi^{j_0}\|_{L^6_{t,x}}$; hence $\phi^{j_0}$ is a maximizer for \eqref{eq:schr-max}. Set $a_n:=h_n^{j_0}\xi_n^{j_0}$. Then the proof of Theorem \ref{thm:airy-max} is complete.
\end{itemize}
\end{proof}

\bibliography{refs}
\bibliographystyle{plain}
\end{document}